\documentclass{article}
\usepackage[utf8]{inputenc} 
\usepackage[english]{babel}
\usepackage{geometry}
\usepackage[numbers]{natbib}
\usepackage{amsmath,amsfonts,amssymb,amsthm,mathrsfs}
\usepackage[hidelinks]{hyperref}
\hypersetup{colorlinks, citecolor=blue, linkcolor=blue, urlcolor=blue}
\usepackage{paralist}
\usepackage{enumitem}
\usepackage{textcase} 
\usepackage{comment, authblk}
\usepackage{tikz}
\usetikzlibrary{arrows}
\usetikzlibrary{decorations.pathreplacing}
\tikzset{mybrace/.style={decoration={brace,raise=1.8mm},decorate}}
\usepackage{titlesec} 
\usepackage{etoolbox} 
\usepackage{float}
\usepackage{graphics}
\usepackage{booktabs}
\usepackage{multirow}
\usepackage{xcolor}
\usepackage{sectsty}
\sectionfont{\bfseries\Large\raggedright}

\makeatletter

\usepackage{color}

\numberwithin{equation}{section}
\allowdisplaybreaks

\theoremstyle{plain}
\newtheorem{theorem}{Theorem}[section]
\newtheorem{definition}[theorem]{Definition}
\newtheorem{lemma}[theorem]{Lemma}

\newtheorem{proposition}[theorem]{Proposition}

\theoremstyle{definition}
\newtheorem{remark}[theorem]{Remark}

\newtheorem{assumption}[theorem]{Assumption}

\titleformat{name=\section}
{\center\small\sc}{\thesection\kern1em}{0pt}{\MakeTextUppercase}
\titleformat{name=\section, numberless}
{\center\small\sc}{}{0pt}{\MakeTextUppercase}

\titleformat{name=\subsection}
{\bf\mathversion{bold}}{\thesubsection\kern1em}{0pt}{}
\titleformat{name=\subsection, numberless}
{\bf\mathversion{bold}}{}{0pt}{}

\titleformat{name=\subsubsection}
{\normalsize\itshape}{\thesubsubsection\kern1em}{0pt}{}
\titleformat{name=\subsubsection, numberless}
{\normalsize\itshape}{}{0pt}{}

\def\note#1{\par\smallskip%
	\noindent\kern-0.01\hsize%
	{\setlength\fboxrule{0pt}\fbox{\setlength\fboxrule{0.5pt}\fbox{%
				\llap{$\boldsymbol\Longrightarrow$ }%
				\vtop{\hsize=0.98\hsize\parindent=0cm\small\rm #1}%
				\rlap{$\enskip\,\boldsymbol\Longleftarrow$}
	}}}%
}

\usepackage{delimset}
\def\given{\mskip 0.5mu plus 0.25mu\vert\mskip 0.5mu plus 0.15mu}
\newcounter{bracketlevel}%
\def\@bracketfactory#1#2#3#4#5#6{%
	\expandafter\def\csname#1\endcsname##1{%
		\global\advance\c@bracketlevel 1\relax%
		\global\expandafter\let\csname @middummy\alph{bracketlevel}\endcsname\given%
		\global\def\given{\mskip#5\csname#4\endcsname\vert\mskip#6}\csname#4l\endcsname#2##1\csname#4r\endcsname#3%
		\global\expandafter\let\expandafter\given\csname @middummy\alph{bracketlevel}\endcsname%
		\global\advance\c@bracketlevel -1\relax%
	}%
}
\def\bracketfactory#1#2#3{%
	\@bracketfactory{#1}{#2}{#3}{relax}{0.5mu plus 0.25mu}{0.5mu plus 0.15mu}
	\@bracketfactory{b#1}{#2}{#3}{big}{1mu plus 0.25mu minus 0.25mu}{0.6mu plus 0.15mu minus 0.15mu}
	\@bracketfactory{bb#1}{#2}{#3}{Big}{2.4mu plus 0.8mu minus 0.8mu}{1.8mu plus 0.6mu minus 0.6mu}
	\@bracketfactory{bbb#1}{#2}{#3}{bigg}{3.2mu plus 1mu minus 1mu}{2.4mu plus 0.75mu minus 0.75mu}
	\@bracketfactory{bbbb#1}{#2}{#3}{Bigg}{4mu plus 1mu minus 1mu}{3mu plus 0.75mu minus 0.75mu}
}
\bracketfactory{clc}{\lbrace}{\rbrace}
\bracketfactory{clr}{(}{)}
\bracketfactory{cls}{[}{]}
\bracketfactory{abs}{\lvert}{\rvert}
\bracketfactory{norm}{\Vert}{\Vert}
\bracketfactory{floor}{\lfloor}{\rfloor}
\bracketfactory{ceil}{\lceil}{\rceil}
\bracketfactory{angle}{\langle}{\rangle}

\newcounter{ctr}\loop\stepcounter{ctr}\edef\X{\@Alph\c@ctr}%
\expandafter\edef\csname s\X\endcsname{\noexpand\mathscr{\X}}
\expandafter\edef\csname c\X\endcsname{\noexpand\mathcal{\X}}
\expandafter\edef\csname b\X\endcsname{\noexpand\boldsymbol{\X}}
\expandafter\edef\csname I\X\endcsname{\noexpand\mathbb{\X}}
\ifnum\thectr<26\repeat

\let\@IE\IE\let\IE\undefined
\newcommand{\IE}{\mathop{{}\@IE}\mathopen{}}
\let\@IP\IP\let\IP\undefined
\newcommand{\IP}{\mathop{{}\@IP}}

\def\^#1{\relax\ifmmode {\mathaccent"705E #1} \else {\accent94 #1}\fi}
\def\~#1{\relax\ifmmode {\mathaccent"707E #1} \else {\accent"7E #1}\fi}
\def\*#1{\relax#1^\ast}
\edef\-#1{\relax\noexpand\ifmmode {\noexpand\bar{#1}} \noexpand\else \-#1\noexpand\fi}
\def\>#1{\vec{#1}}
\def\.#1{\dot{#1}}

\def\atop{\@@atop}

\renewcommand{\leq}{\leqslant}
\renewcommand{\geq}{\geqslant}

\newcommand{\ro}{\mathrm{o}}
\newcommand{\rO}{\mathrm{O}}
\newcommand{\ri}{\mathrm{i}}

\newcommand\indep{\protect\mathpalette{\protect\@indep}{\perp}}
\def\@indep#1#2{\mathrel{\rlap{$#1#2$}\mkern2mu{#1#2}}}

\def\parsetime#1#2#3#4#5#6{#1#2:#3#4}
\def\parsedate#1:20#2#3#4#5#6#7#8+#9\empty{20#2#3-#4#5-#6#7 \parsetime #8}
\def\moddate{\expandafter\parsedate\p\mathrm{d}Ffilemoddate{\jobname.tex}\empty}
\date{}

\theoremstyle{definition}

\theoremstyle{remark}

\theoremstyle{definition}

\theoremstyle{plain}

\theoremstyle{plain}

\theoremstyle{plain}

\theoremstyle{plain}

\allowdisplaybreaks[4]

\makeatother

\providecommand{\conditionname}{Condition}
\providecommand{\definitionname}{Definition}
\providecommand{\lemmaname}{Lemma}
\providecommand{\propositionname}{Proposition}
\providecommand{\remarkname}{Remark}
\providecommand{\corollaryname}{Corollary}
\providecommand{\theoremname}{Theorem}

\title{Tracy-Widom distribution for the edge eigenvalues of elliptical model}

\author[1]{Xiucai Ding \thanks{E-mail: xcading@ucdavis.edu. The author is partially supported by NSF-DMS 2113489. }}
\author[2]{Jiahui Xie \thanks{E-mail:  jiahui.xie@u.nus.edu.}}

\affil[1]{Department of Statistics, University of California, Davis}

\affil[2]{Department of Statistics and Data Science, National University of Singapore}

\begin{document}

\maketitle
\begin{abstract}
In this paper, we study the largest eigenvalues of sample covariance matrices with elliptically distributed data. We consider the sample covariance matrix $Q=YY^*,$ where the data matrix $Y \in \mathbb{R}^{p \times n}$ contains i.i.d. $p$-dimensional observations  
$\mathbf{y}_i=\xi_iT\mathbf{u}_i,\;i=1,\dots,n.$ Here $\mathbf{u}_i$ is distributed on the unit sphere, $\xi_i \sim \xi$ is independent of $\mathbf{u}_i$ and $T^*T=\Sigma$ is some deterministic matrix. Under some mild regularity assumptions of $\Sigma,$ assuming $\xi^2$ has bounded support and certain proper behavior near its edge so that the limiting spectral distribution (LSD) of $Q$ has a square decay behavior near the spectral edge, we prove that the Tracy-Widom law holds for the largest eigenvalues of $Q$ when $p$ and $n$ are comparably large.  
\end{abstract}

\section{Introduction}
Large dimensional sample covariance matrices play important roles in modern statistical learning theory. Understanding the behavior of the largest eigenvalues of the sample covariance matrices is crucial in many statistical techniques, especially the principal component analysis (PCA). Consider $\mathbf{y}_i \in \mathbb{R}^p, 1 \leq i \leq n,$ be a sequence of i.i.d. mean zero random vectors, after properly being scaled, the associated sample covariance matrix can be written as 
\begin{equation*}
Q=\sum_{i=1}^n \mathbf{y}_i \mathbf{y}_i^*.
\end{equation*} 
 Lots of efforts have been made when $\mathbf{y}_i=\Sigma^{1/2} \mathbf{x}_i,$ where $\Sigma$ is some positive define matrix representing the population covariance structures and $\mathbf{x}_i$ contains i.i.d. mean zero entries with variance $n^{-1}.$ In summary, under certain regularity assumptions on $\Sigma$ and moment conditions on $\mathbf{x}_i,$ when $p$ and $n$ are comparably large, after being properly scaled and centered, the largest eigenvalue of $Q$ follows Tracy-Widom (TW) distribution asymptotically. To list but a few, see \cite{BaoPanandZhou2015,DingandYang2018,ding2022tracy,Elkaroui2007,10.1214/aos/1009210544, LS,PillaiandYin2014}.

In this work, we extend this line of research to the setting when $\mathbf{y}_i$ follows elliptical distribution  \cite{fang1990}
\begin{equation}\label{eq_datamodelmodel}
\mathbf{y}_i=\xi_i T\mathbf{u}_i\in\mathbb{R}^p,
\end{equation}
where $\xi_i \in\mathbb{R}$ are i.i.d. random variables, $T^{*}T=\Sigma\in\mathbb{R}^{p\times p}$ is some positive definite deterministic matrix, and $\mathbf{u}_i \in\mathbb{R}^p, 1 \leq i \leq n,$ are distributed on the unit sphere $\mathrm{U}(\mathbb{S}^{p-1})$ that is independent of $\{\xi_i\}$.  We can write the data matrix {$Y=TUD$, where $U=(\mathbf{u}_i)$} and $D$ is a diagonal matrix containing $\{\xi_i\}$ satisfying certain regularity assumptions; see Assumption \ref{assum_D}. Then the sample covariance matrix $Q$ can be rewritten as 
{
\begin{gather}\label{eq_model_ell}
Q=TUD^2U^{*}T^{*}.
\end{gather}
}
We consider the high dimensional setting that for some small fixed constant $0<\tau<1$
\begin{equation}\label{ass1}
 \tau \leq \phi:=\frac{p}{n} \leq \tau^{-1}. 
\end{equation}

It has been shown in \cite{el2009concentration, 8704905}, the empirical spectral distribution (ESD) of $Q$ can be
best formulated by its Stieltjes transform which can be described via a system of two equations. Much less is known about the individual eigenvalues except a recent one \cite{Wen2021}. However, \cite{Wen2021} requires that $\{\xi_i^2\}$ are infinitely divisible so that $D^2$ is almost deterministic. As a consequence, the system of two equations will degenerate to only one. In the current paper, we consider a more challenging setting that $\xi_i^2$ are truly random so that the limiting ESD will be governed by two equations. Under mild conditions, we prove that the largest eigenvalue of $Q$ follows Tracy-Widom distribution asymptotically, after being properly centered and scaled; see Theorem \ref{thm_main_edgeuniversality}. 

This paper is organized as follows. In Section  \ref{sec_defandmainresult}, we provide some useful definitions and state the main results. In Section \ref{sec_preliminary}, we provide some preliminary results. In Section \ref{sec_Proofofmain}, we prove the main results. Finally, the proof of local laws is sketched in Appendix \ref{sec_prooflocalaw}.

\section{Definitions and main results}\label{sec_defandmainresult}
 
\subsection{The model and asymptotic laws}

Throughout the paper, we consider the observations  (\ref{eq_datamodelmodel}) and their associated sample covariance matrix (\ref{eq_model_ell}). Due to rotational invariance, we assume that $\Sigma$ is a diagonal matrix so that  $\Sigma=\operatorname{diag}\left\{ \sigma_1, \cdots, \sigma_p \right\}.$ Moreover, for some small constant $0<\tau<1,$ we assume that 
\begin{equation}\label{ass2}
\tau \leq \sigma_p \leq \sigma_{p-1} \leq \cdots \leq \sigma_2 \leq \sigma_1 \leq \tau^{-1}. 
\end{equation}
For the diagonal matrix $D^2=\operatorname{diag} \left\{ \xi_1^2, \cdots, \xi_n^2 \right\}$ in (\ref{eq_model_ell}) (or equivalently the random variables $\xi_i, 1 \leq i \leq n,$ in (\ref{eq_datamodelmodel})), we impose the following assumptions. Similar conditions have been used in \cite{kwak2021extremal,lee2015edge,lee2016extremal,LSSY}.  
\begin{assumption}\label{assum_D}
We assume $\xi_i^2 \sim \xi^2, 1 \leq i \leq n,$ are i.i.d. random variables. Moreover,  we assume that
 $\xi^2$ has a bounded support on $(0,l]$ for fixed some constant
$l>0.$ Moreover, for some constant $d>-1,$ we assume that  
\begin{equation}\label{assum_tailrate}
\mathbb{P}( l-\xi^2 \leq x) \asymp x^{d+1}.
\end{equation}
\end{assumption}

Then we prepare some notations. For the sample covariance matrix $Q$ in (\ref{eq_model_ell}) and its companion $\mathcal Q,$ 
{
\begin{equation}\label{eq_gram}
\mathcal Q:=D U^*T^* T U D \equiv DU^*\Sigma UD,
\end{equation}  
}
their empirical spectral distributions (ESD) are defined as $\mu_{Q}:=p^{-1}\sum_{i=1}^p \delta_{\lambda_i(Q)}$ and $\mu_{\mathcal{Q}}:=n^{-1}\sum_{j=1}^n \delta_{\lambda_j(\mathcal{Q})},$ respectively.
Correspondingly, the Stieltjes transforms are denoted as
\begin{equation}\label{eq_mq}
m_{Q}:=\int\frac{1}{x-z}\mu_{Q},\quad m_{\mathcal{Q}}:=\int\frac{1}{x-z}\mu_{\mathcal{Q}}, \ z \in \mathbb{C}_+. 
\end{equation}
Since $Q$ and $\mathcal{Q}$ share the same non-trivial eigenvalues, it suffices to study $\mu_{Q}.$  To characterize the limit of $\mu_{Q},$ we consider a system of equations. 
\begin{definition}[System of consistent equations]\label{defn_couplesystem} For $z \in \mathbb{C}_+,$ we define the triplets $(m_{1n}, m_{2n}, m_n) \in \mathbb{C}^3_+,$  via the following system of equations.
\begin{gather}\label{eq_systemequationsm1m2elliptical}
    m_{1n}(z)=\frac{1}{p}\sum_{i=1}^p\frac{\sigma_i}{-z(1+\sigma_i  m_{2n}(z))},\quad m_{2n}(z)=\frac{1}{p}\sum_{i=1}^n\frac{\xi_i^2}{-z(1+\xi^2_im_{1n}(z))}\\
    m_{n}(z)=\frac{1}{p}\sum_{i=1}^p\frac{1}{-z(1+\sigma_i  m_{2n}(z))}. \nonumber
\end{gather}
Moreover, we define the 
the triplets $(m_{1n,c}, m_{2n,c}, m_{n,c}) \in \mathbb{C}^3_+,$  via the following system of equations.
\begin{gather}\label{eq_systemequationsm1m2ellipticallimit}
    m_{1n,c}(z)=\frac{1}{p}\sum_{i=1}^p\frac{\sigma_i}{-z(1+\sigma_i  m_{2n,c}(z))},\quad m_{2n,c}(z)=\phi^{-1}\int_0^l \frac{s}{-z(1+s m_{1n,c}(z))} \mathrm{d}F(s)\\
    m_{n,c}(z)=\frac{1}{p}\sum_{i=1}^p\frac{1}{-z(1+\sigma_i  m_{2n,c}(z))}, \nonumber
\end{gather}
where $F(s)$ is the cumulative distribution (CDF) of $\xi^2$ and recall $\phi$ in (\ref{ass1}). 
\end{definition}
We point out that according to Lemma \ref{defn_Omega} below, (\ref{eq_systemequationsm1m2ellipticallimit}) can be regarded as an asymptotic deterministic equivalent of (\ref{eq_systemequationsm1m2elliptical}). To avoid repetitions, we summarize some assumptions as follows. 
\begin{assumption}\label{assum_techincial}
We assume (\ref{eq_datamodelmodel}), (\ref{ass1}), (\ref{ass2}) and Assumption \ref{assum_D} hold. 
\end{assumption}


\begin{theorem}\label{lem_solutionsystem}
Suppose Assumption \ref{assum_techincial} holds. Then conditional on some event $\Omega \equiv \Omega_n$ (defined in Lemma \ref{defn_Omega}) that $\mathbb{P}(\Omega)=1-\ro(1),$ for any $z \in \mathbb{C}_+,$ when $n$ is sufficiently large, there exists a unique solution $(m_{1n}(z), m_{2n}(z), m_{n}(z)) \in \mathbb{C}_+^3$ to the system of equations in (\ref{eq_systemequationsm1m2elliptical}). Moreover, $m_n(z)$ is the Stieltjes transform of some probability measure $\mu \equiv \mu_n$ defined on $\mathbb{R}$ which can be obtained using the inversion formula and has a continuous derivative $\rho$ on $(0, \infty)$. Similar results hold for $(m_{1n,c}, m_{2n,c}, m_{n,c})$ in (\ref{eq_systemequationsm1m2ellipticallimit}) unconditionally. 
\end{theorem}
\begin{proof}
The proofs can be obtained by following lines of the arguments  of \cite[Theorem 2]{el2009concentration} and \cite[Theorem 1]{PaulandSilverstein2009}, or see \cite[Theorem 2.4]{ding2021spiked} verbatim. We omit the details. 
\end{proof}

Thanks to Theorem \ref{lem_solutionsystem}, it is easy to see that the study of the system of equations, for example (\ref{eq_systemequationsm1m2elliptical}),  can be reduced to the analysis of 
\begin{equation}\label{eq_functionFequal}
F_p(m_{1n}(z),z)=0, \quad z\in\mathbb{C}_{+},
\end{equation}
where $F_p(\cdot, \cdot)$ are defined as follows
\begin{equation}\label{eq:F(m,z)}
    F_p(m_{1n}(z),z)
    =\frac{1}{p}\sum_{i=1}^p\frac{\sigma_i}{-z+\frac{\sigma_i}{p}\sum_{j=1}^n\frac{\xi^2_j}{1+\xi^2_j m_{1n}(z)}}-m_{1n}(z).
\end{equation} 
Moreover, the right-most edge of the support can also be characterized using (\ref{eq:F(m,z)}). Denote $\rho, \rho_1, \rho_2, \rho_c, \rho_{1c}$ and $\rho_{2c}$ as the density functions associated with the Stieltjes transforms $m_n, m_{1n}, m_{2n}, m_{n,c}, m_{1n,c}$ and $m_{2n,c}$ in Theorem \ref{lem_solutionsystem} via the inversion formula. 
Let $\lambda_+$ be the right-most edge of the support of $\rho$ and $L_+$ be that of $\rho_c.$ Then we have that 
\begin{lemma}\label{lem_edge} Suppose the assumptions of Theorem \ref{lem_solutionsystem} hold. Then condition on the event $\Omega,$ we have that 
\begin{gather*}
    \operatorname{supp} \rho\cap(0,\infty)=\operatorname{supp}\rho_{1(2)}\cap(0,\infty).
\end{gather*}
Moreover, $(x,y)=(m_{1n}(\lambda_{+}),\lambda_{+})$ is the real solution that satisfies the following equations
\begin{gather}\label{eq: def of lambda+}
    F_p(x,y)=0,\quad \frac{\partial F_p}{\partial x}(x,y)=0.
\end{gather}
Similar results hold for $\rho_c, \rho_{1(2),c}$ and $L_+$ unconditionally when $F_p$ is replaced by $F_{p,c}$ defined using (\ref{eq_systemequationsm1m2ellipticallimit}) that 
\begin{equation}\label{eq:generalfunctionlimit}
F_{p,c}(x,y)=\frac{1}{p} \sum_{i=1}^p  \frac{\sigma_i}{-y+ \phi^{-1} \sigma_i \int \frac{s}{1+s x} \mathrm{d} F(s)}-x.
\end{equation} 
\end{lemma}
\begin{proof}
See Lemma 2.5 of \cite{yang2019edge}. 
\end{proof}

To study the edge behaviors, we need the following assumption which guarantees a regular square-root behavior of the spectral densities near the edges.

\begin{assumption}\label{assum_techniqueSigma}
Suppose that for some constant $\tau>0,$
    \begin{gather}\label{eq: reg of sigma}
    |1+\sigma_1m_{2n,c}(L_{+})|\ge\tau.
\end{gather}
Moreover, for $\xi^2$ in Assumption \ref{assum_D}, we assume that either (1). $-1<d \leq 1$ or (2). $d>1$ and $\phi^{-1}<\vartheta$, where $\vartheta \equiv \vartheta(\xi^2)$ is defined as 
\begin{equation*}
\vartheta=\frac{\phi^{-1}}{p} \sum_{i=1}^p \frac{\sigma_i^2 \upsilon_1}{(L_+-\sigma_i \upsilon_2)^2}, \ \text{where} \ \upsilon_1:=\phi^{-1} \int_0^l \frac{l^2 s^2}{(l-s)^2} \mathrm{d} F(s) \ \text{and} \ \upsilon_2:=\phi^{-1} \int_0^l \frac{ls}{l-s} \mathrm{d} F(s).  
\end{equation*} 
\end{assumption}

\begin{remark}
As will be seen in Theorem \ref{thm_main_asymptotic laws} below, the conditions in Assumption \ref{assum_techniqueSigma} ensure that   $\rho_c, \rho_{1,c}, \rho_{2,c}$ have a square root decay behavior near $L_+.$ Furthermore, it also implies that  $\rho, \rho_1, \rho_2$ have a square root decay behavior near $\lambda_+$ on $\Omega.$ We remark that the square root behavior of the LSD is generally believed to be necessary condition for the appearance of the Tracy-Widom law in the asymptotic limit. For example, if the LSD has a linear behavior near the edge, then the corresponding asymptotics will be very different \cite{Ding2023,kwak2021extremal,lee2015edge,lee2016extremal}. 
\end{remark}

\subsection{Main results}\label{sec_mainresultsection}
The main theorem of the paper is stated as follows. Let $\gamma_0 \equiv \gamma_0(n)$ be defined as follows 
\begin{equation}\label{eq_gamma0definition}
\gamma_0^3=\frac{-2\partial_y F_{p,c}(m_{1n,c}(L_{+}),L_{+})}{\partial^2_x F_{p,c}(m_{1n,c}(L_{+}),L_{+})}\left(\phi^{-1} \int_0^l\frac{s}{L_{+}(1+s m_{1n,c}(L_{+}))^2} \mathrm{d}F(s)\right)^2.
\end{equation}
Let $\lambda_1 \geq \lambda_2 \geq \lambda_{\min\{p,n\}}>0$ be the non-zero eigenvalues of $Q.$
\begin{theorem}\label{thm_main_edgeuniversality}
Suppose Assumptions \ref{assum_techincial} and \ref{assum_techniqueSigma} hold. Then we have that for all $x\in\mathbb{R}$ 
        \begin{gather}\label{eq_resulteq}
        \lim_{n\rightarrow\infty}\mathbb{P}\big(\gamma_0 n^{2/3}(\lambda_1-\lambda_{+})\le x\big)=F_1(x),
    \end{gather}
where $\lambda_{+}$ is defined in Lemma \ref{lem_edge}  via \eqref{eq: def of lambda+}, $\gamma_0$ is defined in (\ref{eq_gamma0definition}) and $F_1$ is the type-1 Tracy-Widom cumulative distribution function.
\end{theorem}

\begin{remark}\label{rmk_orderorder}
Several remarks are in order. First, in (\ref{eq_resulteq}), as discussed in Remark 2.9 of \cite{yang2019edge}, $\gamma_0$ is deterministic and $\gamma_0 \asymp 1$. In addition, $\lambda_+$ is random in general. Second, similar to Theorem III.2 of \cite{ding2022tracy}, we can generalize (\ref{eq_resulteq}) to multiple edge eigenvalues. Recall the Gaussian orthogonal ensemble (GOE) refers to symmetric random matrices of the form $(X+X^*)/\sqrt{2},$ where $X$ is a $n \times n$ matrix with i.i.d. Gaussian entries with mean zero and variance $n^{-1}.$ Let $\{\mu_i^{\text{GOE}}\}$ be the eigenvalues of GOE in the decreasing order, then for any fixed $k \in \mathbb{N},$ we can prove that  for all $(x_1, x_2,\cdots, x_k) \in \mathbb{R}^k$
 \begin{gather*}
        \lim_{n \rightarrow\infty}\mathbb{P}\left[ \left(\gamma_0 n^{2/3}(\lambda_i-\lambda_{+})\le x_i \right)_{1 \leq i \leq k}\right]=\lim_{n \rightarrow\infty}\mathbb{P}\left[ \left(n^{2/3}(\mu^{\text{GOE}}_i-2)\le x_i \right)_{1 \leq i \leq k}\right].
    \end{gather*}
Third, we compare our results with two related works. On the one hand, in \cite{Ding2023}, the authors also consider the edge eigenvalues of (\ref{eq_model_ell}) with a focus that the LSD can be either unbounded or have a linear decay behavior near the edge. Therefore, the asymptotics are not Tracy-Widom in general.  On the other hand, in \cite{Wen2021}, the authors assumed that $\xi^2$ had an infinite divisible representation so that the system of equations in Definition \ref{defn_couplesystem} reduces to only one. As a result, $\xi_i$'s in (\ref{eq_datamodelmodel}) are almost deterministic so that the proof follows straightforwardly from those of the sample covariance matrix with i.i.d. entries or Wigner matrices, for example, see \cite{BaoPanandZhou2015,DingandYang2018,LeenandYin2012,PillaiandYin2014}. 
\end{remark}

The proof of Theorem \ref{thm_main_edgeuniversality} relies on the Green function comparison approach as in \cite{BaoPanandZhou2015,DingandYang2018,LSSY,PillaiandYin2014}, except that we will conduct an expansion up to the order of four as in \cite{bao2022extreme}. In Corollary 1 of \cite{ding2022tracy}, it was shown that when {$U$} in (\ref{eq_model_ell}) is replaced by a Gaussian matrix, its largest eigenvalue will follow TW law asymptotically. More specifically, conditional on some realization of $D$ which happens with high probability (c.f. Lemma \ref{defn_Omega}), consider $Q^G=TZD^2 Z^* T^*,$ where $Z=(\mathbf{z}_i)$ and $\mathbf{z}_i\overset{\mathrm{i.i.d.}}{\sim} \mathcal{N}(0,p^{-1}I_{p\times p}),$ it has been proved that $ \gamma_0 n^{2/3} (\lambda_1(Q^G)-\lambda_+)$ follows TW law asymptotically. Based on this, it suffices to conduct the comparison argument between $Q$ and $Q^G$ which relies on two important ingredients. First, in Lemma \ref{thm_main_asymptotic laws}, we conduct the local analysis for the system of equations in Definition \ref{defn_couplesystem} from which          we observe the square root behavior near the edge. Second, for comparison, in Theorem \ref{thm_main_locallaws}, we prove the local laws for the following matrices 
{
\begin{gather}\label{eq_samplecov_t}
S:= VV^* \equiv TWD^2W^{*}T^{*},\quad \mathcal{S}:= V^*V  \equiv DW^{*}\Sigma WD, \ \  V:=TWD, 
\end{gather}
where $W=(\mathbf{w}_i)$ with $\mathbf{w}_i=\mathbf{u}_i\;\text{or}\;\mathbf{z}_i.$ We note that if {$W=U$} or $W=Z$, then $S\equiv Q$ or $S\equiv Q^G$, respectively.}

\section{Preliminaries}\label{sec_preliminary}
We provide some preliminary results. First, since $D^2$ is random, it is more convenient to fixed some realization with certain properties which happens with high probability. The following lemma provides such probability events. 
\begin{lemma}\label{defn_Omega}
Under Assumptions \ref{assum_D}, denote $\Omega \equiv \Omega_n$ be the event on $\{\xi_i^2\}$ so that the following conditions hold:
    \begin{equation}\label{eq_property}
\begin{aligned}
& n^{-1/(d+1)}\log^{-1}n<l-\xi^2_{(1)}<n^{-1/(d+1)}\log n, \ \ \xi^2_{(1)}-\xi^2_{(2)}> n^{-1/(d+1)}\log^{-1}n, \\
&  \left|\frac{1}{n}\sum_{i=1}^n\frac{\xi^2_i}{1+\xi^2_{i}m_{1 n}(z)}-\int\frac{t}{1+tm_{1 n}(z)}\mathrm{d} F(t)\right|\le\frac{Cn^{\epsilon}}{\sqrt{n}},
\end{aligned}
\end{equation}
where $C>0$ is some generic constant and $\epsilon>0$ is some arbitrarily small constant. Then  we have that $\mathbb{P}(\Omega)=1-\mathrm{o}(1).$
\end{lemma} 
\begin{proof}
See Lemma A.11 of \cite{Ding2023}. 
\end{proof}

Second, we summarize the properties of the Stieltjes transforms. For some fixed (small) constants $\mathsf{c}, \mathsf{C}, \epsilon_e>0,$ define the spectral parameter sets 
\begin{equation}\label{eq: spectraldomainD}
\mathbf{D} \equiv \mathbf{D}(\mathtt c,\mathtt C,\varepsilon_e):=\left\{z=E+\ri\eta \in \mathbb{C}_+: \lambda_{+}-\mathtt c\le E\le \lambda_{+}+\mathtt C,\; p^{-1+\varepsilon_e}\le\eta\le \mathtt C\right\},
\end{equation}
and 
\begin{equation}\label{eq: spectraldomainDprime}
\mathbf{D}_0 \equiv \mathbf{D}(\mathtt c,\mathtt C,\infty):=\left\{z=E+\ri\eta \in \mathbb{C}_+: \lambda_{+}-\mathtt c\le E\le \lambda_{+}+\mathtt C, \;0<\eta\le \mathtt C\right\}.
\end{equation}
\begin{lemma}\label{thm_main_asymptotic laws}
    Suppose Assumptions \ref{assum_techincial} and \ref{assum_techniqueSigma} hold. When $n$ is sufficiently large, for any realization $\{\xi^2_i\} \in \Omega$ defined in Lemma \ref{defn_Omega}, we have that  
    \begin{itemize}
      \item [(1)] For all $z\in\mathbf{D}$ with $z\rightarrow\lambda_{+}$,
    \begin{gather}\label{eq: asymlaws_main_edge}
        m_{n}(\lambda_{+})-m_{n}(z) \asymp \sqrt{|\lambda_{+}-z|}.
    \end{gather}
    Consequently, for small constant $\kappa>0$
     \begin{gather}\label{eq: asymlaws_main_squareroot}
        \rho(\lambda_{+}-\kappa) \asymp \sqrt{\kappa}. 
    \end{gather}
    The results also apply to $m_{1(2)n}$ and $\rho_{1(2)}.$
    \item [(2)]  For all $z\in\mathbf{D}_0$, 
    \begin{gather}\label{eq: asymlaws_main_estm12n}
        |m_{n}(z)|\sim1,\quad \operatorname{Im}m_{n}(z)\sim
        \begin{cases}
        \eta/\sqrt{\kappa+\eta},& \quad \operatorname{if} \quad E\ge\lambda_{+}\\
        \sqrt{\kappa+\eta}, & \quad \operatorname{if} \quad E\le\lambda_{+}.
        \end{cases}
    \end{gather}
      The results also apply to $m_{1(2)n}.$ 
    \item [(3)]  There exists an constant $\tau^{\prime}>0$, such that  for any $z\in\mathbf{D}_0$
    \begin{gather}\label{eq: asymlaws_main_lowerbound}
        \min_{1 \leq i \leq p}|1+\sigma_im_{2n}(z)|\ge\tau^{\prime},\quad  \min_{1 \leq j \leq n}|1+\xi^2_{j}m_{1n}(z)|\ge\tau^{\prime}.
    \end{gather}
    \end{itemize}
    Finally, all the results in (\ref{eq: asymlaws_main_edge})--(\ref{eq: asymlaws_main_lowerbound}) also hold for $m_{n,c}, m_{1(2)n,c}, \rho_c, \rho_{1(2)c}$ unconditionally without fixing the realization. 
\end{lemma}
\begin{proof}
See Lemma A.4 of \cite{Ding2023}. 
\end{proof}

Third, we provide the local laws. The following notion will be used in the statements. 

\begin{definition}[Stochastic domination] Let
\[A=\left(A^{(n)}(u):n\in\mathbb{N}, u\in U^{(n)}\right),\hskip 10pt B=\left(B^{(n)}(u):n\in\mathbb{N}, u\in U^{(n)}\right),\]
be two families of nonnegative random variables, where $U^{(n)}$ is a possibly $n$-dependent parameter set. We say $A$ is stochastically dominated by $B$, uniformly in $u$, if for any fixed (small) $\epsilon>0$ and (large) $\xi>0$, 
\[\sup_{u\in U^{(n)}}\mathbb{P}\left(A^{(n)}(u)>n^\epsilon B^{(n)}(u)\right)\le n^{-\xi},\]
for large enough $n \ge n_0(\epsilon, \xi)$, and we shall use the notation $A\prec B$ or $A=\mathrm{O}_\prec(B)$. Throughout this paper, the stochastic domination will always be uniform in all parameters that are not explicitly fixed, such as the matrix indices and the spectral parameter $z$.  
\end{definition} 


Recall (\ref{eq_samplecov_t}). For $z=E+\ri \eta \in \mathbb{C}_+,$ denote the resolvents 
\begin{equation}\label{eq_resolvents}
G(z)=(S-zI)^{-1} \in \mathbb{R}^{p \times p},\quad \mathcal{G}(z)=(\mathcal{S}-zI)^{-1} \in \mathbb{R}^{n \times n}.
\end{equation}
Moreover, we define $m(z):=p^{-1}{\rm tr}(G(z))$ and $m_1(z):=p^{-1}\operatorname{tr}(G(z)\Sigma).$

\begin{theorem}\label{thm_main_locallaws}
    Suppose Assumptions \ref{assum_techincial} and \ref{assum_techniqueSigma} hold.  When $n$ is sufficiently large, for any realization $\{\xi^2_i\} \in \Omega$ defined in Lemma \ref{defn_Omega}, we have that for $z \in \mathbf{D}$ in (\ref{eq: spectraldomainD}) uniformly,  
  \begin{gather}\label{eq_entrywise}
        \max_{1 \leq i,j \leq n}|\mathcal{G}+z^{-1}(1+m_{1n}(z)D^2)^{-1}|_{ij}\prec \sqrt{\frac{\operatorname{Im} m_{1n}(z)}{ p \eta}}+\frac{1}{p \eta}.
    \end{gather}    
Moreover,     
       \begin{gather}\label{eq_averagedone}
        {|m_{1n}(z)-m_1(z)|+|m_n(z)-m(z)|\prec (p\eta)^{-1}.}
    \end{gather}
Finally,  for any $z\in\mathbf{D}\cap\{z=E+\mathrm{i}\eta:E\ge\lambda_{+}, \ p\eta\sqrt{\kappa+\eta}\ge p^{\varepsilon_e}\}$ uniformly, 
    \begin{gather}\label{eq_averagedtwo}
        {|m_{1n}(z)-m_1(z)|+|m_n(z)-m(z)|\prec \frac{1}{p(\kappa+\eta)}+\frac{1}{(p\eta)^2\sqrt{\kappa+\eta}}. }
    \end{gather}

\end{theorem}

\section{Main technical proofs}\label{sec_Proofofmain}
As discussed in the end of Section \ref{sec_mainresultsection}, our results rely on a comparison argument. The core is to prove the following argument. {Let $\widetilde{G}(z):=(Q^G-zI)^{-1}$ be resolvent of $Q^G$ and $\widetilde{m}(z):=p^{-1}\operatorname{tr}\widetilde{G}(z)$. }
\begin{proposition}\label{prop_greenfuncomp}
Suppose the assumptions of Lemma \ref{thm_main_asymptotic laws} hold and $F:\mathbb{R}\rightarrow\mathbb{R}$ is a function whose derivatives satisfy
    \[
		\sup_{x\in\mathbb{R}}|F^{(l)}(x)|(1+|x|)^{-C_{1}}\le C_{1},\qquad l=1,2,3,4, 
		\]
for	 some constant $C_{1}>0$. Then for any sufficiently small constant $\epsilon>0$ and  for any real
		numbers $E,E_{1}$ and $E_{2}$ satisfying
		\begin{equation}\label{eq_edgeepsilon}
		|E-\lambda_{+}| \le n^{-2/3+\epsilon}, \ |E_{1}-\lambda_{+}| \le n^{-2/3+\epsilon}, \ |E_{2}-\lambda_{+}|\le n^{-2/3+\epsilon},
		\end{equation}
		and $\eta_0=n^{-2/3-\epsilon}$,  there exist some constants $c_1,C>0$ 
  \begin{gather}\label{eq_greenfuncomp_ineq1}
      |\mathbb{E}F(n\eta_0\operatorname{Im} m(z))-\mathbb{E}F(n\eta_0\operatorname{Im}{\widetilde{m}(z))}|\le n^{-c_1+C {\epsilon}},\qquad z=E+\mathrm{i}\eta_0,
  \end{gather}
  and
  \begin{gather}\label{eq_greenfuncomp_ineq2}
      \Big|\mathbb{E}F\Big(\int_{E_{1}}^{E_{2}}n\operatorname{Im} m(y+\ri \eta_0){\rm d}y\Big)-\mathbb{E}F\Big(\int_{E_{1}}^{E_{2}}n\operatorname{Im} {\widetilde{m}(y+\ri \eta_0)}{\rm d}y\Big)\Big|\le n^{-c_1+C {\epsilon}}. 
  \end{gather}
\end{proposition}

Armed with Proposition \ref{prop_greenfuncomp}, we are able to prove the main result, Theorem \ref{thm_main_edgeuniversality}.
\begin{proof}[\bf Proof of Theorem \ref{thm_main_edgeuniversality}]
Using Proposition \ref{prop_greenfuncomp}, following lines of the arguments of Corollary 4.2 of \cite{PillaiandYin2014} or Lemmas 5.1 and 5.2 of \cite{DingandYang2018}, we can show that for any fixed realization $\{\xi_i^2\} \in \Omega$ in Lemma \ref{defn_Omega} and any $x \in \mathbb{R},$ 
\begin{gather}\label{eq_kkkooo}
    \lim_{n\rightarrow\infty}\mathbb{P}(n^{2/3}(\lambda_1(Q)-\lambda_{+})\le x)=\lim_{n\rightarrow\infty}\mathbb{P}(n^{2/3}(\lambda_1(Q^G)-\lambda_{+})\le x). 
\end{gather}
It has been proved in Corollary 1 of \cite{ding2022tracy} and Section 2.3 \cite{yang2019edge} that  
    \begin{gather} \label{eq_kkkooo22}
        \lim_{n\rightarrow\infty}\mathbb{P}(\gamma_0n^{2/3}(\lambda_1(Q^G)-\lambda_{+})\le x)=F_1(x).
    \end{gather}
Moreover, as discussed in Remark \ref{rmk_orderorder}, we have that $\gamma_0 \asymp 1.$ For the unconditional setting, we see from Lemma \ref{defn_Omega} that $\mathbb{P}(\Omega)=1-\ro(1)$. Therefore, Theorem \ref{thm_main_edgeuniversality} can be proved using(\ref{eq_kkkooo}) and (\ref{eq_kkkooo22}).
\end{proof} 

The rest of the section leaves to the proof of Proposition \ref{prop_greenfuncomp}.

\subsection{Proof of Proposition \ref{prop_greenfuncomp}}
We first prepare some notations.  For the data matrix $V$ in (\ref{eq_samplecov_t}) and the index set ${\cal I}=\{1,\dots,n\}$, given $\mathcal{T}\subset{\cal I}$, we introduce
	the notation $V^{(\mathcal{T})}$ to denote the $p \times(n-|\mathcal{T}|)$ minor of $V$ obtained from removing all the $i$th columns of $V$ for $i\in \mathcal{T}$ and keep the original indices of $V$.  For convenience, we briefly write $\{i\}$, $\{i,j\}$ and $\{i,j\}\cup \mathcal{T}$ as $(i)$, $(ij)$
	and $(ij\mathcal{T})$ respectively. Correspondingly, we denote 
the matrices 	$S^{(\mathcal{T})}=(V^{(\mathcal{T})}) (V^{(\mathcal{T})})^*,\ {\mathcal{S}}^{(\mathcal{T})}=(V^{(\mathcal{T})})^* (V^{(\mathcal{T})})$ 
	and their associated resolvents as $G^{(\mathcal{T})}(z)=(S^{(\mathcal{T})}-zI)^{-1}$ and ${\cal G}^{(\mathcal{T})}(z)=({\mathcal{S}}^{(\mathcal{T})}-zI)^{-1},$ respectively. 

\begin{proof}[\bf Proof of Proposition \ref{prop_greenfuncomp}]
Due to similarity, we only prove (\ref{eq_greenfuncomp_ineq1}). We consider the following Linderberg replacement for {$U$}. For $\gamma=0,\dots,n$, let {$W_{\gamma}$} be the matrix whose {first $\gamma$ columns} are the same as those of {$Z$} and the remaining $n-\gamma$ columns are the same as those of {$U$}. Then, it is easy to see {$W_0=U$ and $W_n=Z$}. Denote $G_{\gamma}$, $\mathcal{G}_{\gamma}$ as the Green functions of {$TW_{\gamma}D^2W_{\gamma}^{*}T^{*}$ and $DW_{\gamma}^{*}T^{*}TW_{\gamma}D$} respectively, and $m_{p,\gamma}=p^{-1}\operatorname{tr}G_{\gamma}$, $m_{n,\gamma}=p^{-1}\operatorname{tr}\mathcal{G}_{\gamma}$. We now rewrite (\ref{eq_greenfuncomp_ineq1}) as the following telescoping summation
    \begin{gather*}
   \mathbb{E}\big(F(n\eta_0\operatorname{Im}m(z))\big)-\mathbb{E}\big(F(n\eta_0\operatorname{Im}\widetilde{m}(z))\big)=\sum_{\gamma=1}^n\Big(\mathbb{E}\big(F(n\eta_0\operatorname{Im}m_{p,\gamma-1}(z))\big)-\mathbb{E}\big(F(n\eta_0\operatorname{Im}m_{p,\gamma}(z))\big)\Big).
    \end{gather*}
    It suffices to prove that for all $1 \leq \gamma \leq n$
    \begin{gather}\label{eq_greenfuncomp_diffdecomp}
        |\mathbb{E}\big(F(n\eta_0\operatorname{Im}m_{p,\gamma-1}(z))\big)-\mathbb{E}\big(F(n\eta_0\operatorname{Im}m_{p,\gamma}(z))\big)|\le n^{-1-c+C\epsilon}.
    \end{gather}
Following \cite{BaoPanandZhou2015,DingandYang2018,PillaiandYin2014}, by introducing $m_{*,\gamma}=p^{-1}\operatorname{tr}G_{\gamma}^{(\gamma)}$, it suffices to prove  
    \begin{gather}\label{eq_greenfuncomp_diffdecomp2}
    \begin{split}
    \Big|\mathbb{E}\big[F(n\eta_0\operatorname{Im}m_{p,\gamma-1}(z))\big]-\mathbb{E}\big[F(n\eta_0\operatorname{Im}m_{*,\gamma}(z))\big]-\Big(\mathbb{E}\big[F(n\eta_0\operatorname{Im}m_{p,\gamma}(z))\big]-\mathbb{E}\big[F(n\eta_0\operatorname{Im}m_{*,\gamma}(z))\big]\Big)\Big|\le n^{-1-c_1+C\epsilon}.
    \end{split}
    \end{gather}
Throughout the proof, for notational convenience, in (\ref{eq_samplecov_t}), we set $Y=TUD$ if $W=U$ and $Y^G=TZD$ if $W=Z.$ Moreover,  we denote the $\gamma$-th column of $Y$ and $Y^G$ as $\mathbf{y}_{\gamma}$ and $\widetilde{\mathbf{y}}_{\gamma}$ respectively. Our arguments rely on the following estimates whose proof will be given in Section \ref{sec_append_prior}.

\begin{lemma}\label{lem_somekeyestimates} For $\epsilon$ in (\ref{eq_edgeepsilon}), we have that 
   \begin{gather}
        |\mathbf{y}^{*}_{\gamma}(G_{\gamma}^{(\gamma)})^2\mathbf{y}_{\gamma}|\prec n^{1/3+\epsilon},\label{eq_greenfuncomp_priorest1}\\
       {{|[G_{\gamma}^{(\gamma)}]_{ii}+z^{-1}(1+\sigma_im_{2n}(z))^{-1}|}\prec n^{-1/3+\epsilon},}\quad |[G_{\gamma}^{(\gamma)}]_{ij}|\prec n^{-1/3+\epsilon},\; 1 \leq i\neq j \leq p,\label{eq_greenfuncomp_priorest2}\\
        |[(G_{\gamma}^{(\gamma)})^2]_{ij}|\prec n^{1/3+2\epsilon}, \ 1 \leq i\neq j \leq p. \label{eq_greenfuncomp_priorest3}
    \end{gather}
\end{lemma}    


Set $\mathcal{X}:=-\mathbf{y}_{\gamma}^{*}(G^{(\gamma)}_{\gamma})^2\mathbf{y}_{\gamma}$ and $\mathcal{Y}:=-\mathbf{y}_{\gamma}^{*}G^{(\gamma)}_{\gamma}\mathbf{y}_{\gamma}$ with their counterparts $\widetilde{\mathcal{X}}:=-\widetilde{\mathbf{y}}_{\gamma}^{*}(G^{(\gamma)}_{\gamma})^2\widetilde{\mathbf{y}}_{\gamma}$ and $\widetilde{\mathcal{Y}}:=-\widetilde{\mathbf{y}}_{\gamma}^{*}G_{\gamma}^{(\gamma)}\widetilde{\mathbf{y}}_{\gamma}$.  Using Lemma \ref{lem: basictool_largedevia}, Theorem \ref{thm_main_locallaws}, Lemmas \ref{lem: basictool_greenfunciden} and \ref{lem_somekeyestimates}, we have that 
\begin{equation}\label{eq_fundementalestimate}
|\mathcal{X}(\mathcal{Y}+\xi^2_{\gamma}m_{1n}(z))^s|\prec n^{-(s-1)/3+(s+1)\epsilon}.
\end{equation}
Using (\ref{eq_fundementalestimate}), (\ref{eq: asymlaws_main_estm12n}) and decomposing $\mathcal{Y}=\mathcal{Y}+\xi^2_{\gamma}m_{1n}(z)-\xi^2_{\gamma}m_{1n}(z)$, we conclude from Binomial theorem that for some constants $C_s=\rO(1), 0 \leq s \leq 2$ 
\begin{gather}\label{eq_expanforXY}
\begin{split}
    \operatorname{Im}\big(\sum_{k\ge 0}\mathcal{X}\mathcal{Y}^k\big)&=\operatorname{Im}\big(\sum_{0\le s\le 2}C_s\mathcal{X}(\mathcal{Y}+\xi^2_{\gamma}m_{1n}(z))^s\big)+\operatorname{Im}\big(\sum_{l\ge 0}(-\xi^2_{\gamma}m_{1n}(z))^l\mathcal{X}\big)+\rO_{\prec}(n^{-2/3+4\epsilon})\\
    &=\operatorname{Im}\big(\sum_{0\le s\le 2}C_s\mathcal{X}(\mathcal{Y}+\xi^2_{\gamma}m_{1n}(z))^s\big)+\operatorname{Im}\frac{\mathcal{X}}{1+\xi^2_{\gamma}m_{1n}(z)}+\rO_{\prec}(n^{-2/3+4\epsilon}).
\end{split}
\end{gather}
{where we used the fact $m_{1n}(\lambda_{+})>-l^{-1}$ as in Lemma A.4 of \cite{Ding2023} so that $|\xi^2_{\gamma}m_{1n}(z)|<1$.} Since $G_{\gamma-1}^{(\gamma)}=G_{\gamma}^{(\gamma)},$ we have the following resolvent expansion 
\begin{gather}\label{eq_greenfuncomp_resolventexp}
        G_{\gamma-1}=G_{\gamma}^{(\gamma)}-G_{\gamma-1}\mathbf{y}_{\gamma}\mathbf{y}_{\gamma}^{*}G_{\gamma}^{(\gamma)},
    \end{gather}
    which results in 
    \begin{gather}\label{eq_priorest_greenfuncomp}
    n\eta_0|m_{p,\gamma-1}-m_{*,\gamma}|=n\eta_0|\sum_{k\ge0}\mathcal{X}\mathcal{Y}^k|=\phi^{-1}\eta_0|z[\mathcal{G}_{\gamma-1}]_{\gamma\gamma}\mathbf{y}_{\gamma}^{*}(G_{\gamma}^{(\gamma)})^2\mathbf{y}_{\gamma}|. 
\end{gather}
Together with Lemma \ref{lem_somekeyestimates} and Theorem \ref{thm_main_locallaws}, we see that  $n\eta_0|m_{p,\gamma-1}-m_{*,\gamma}| \prec n^{-1/3+\epsilon}.$ Using the above bound and (\ref{eq_fundementalestimate}), combining  \eqref{eq_expanforXY}, \eqref{eq_priorest_greenfuncomp} and Lemma \ref{lem_somekeyestimates}, we see that for some constant $C>0,$ the first difference in \eqref{eq_greenfuncomp_diffdecomp2} can be expanded as 
\begin{align}
            &F(n\eta_0\operatorname{Im}m_{p,\gamma-1})-F(n\eta_0\operatorname{Im}m_{*,\gamma})
           \nonumber \\ 
            &=F^{(1)}(n\eta_0\operatorname{Im}m_{*,\gamma})\times \phi^{-1}\eta_0\operatorname{Im}\big(\sum_{0\le s\le 2}C_s\mathcal{X}(\mathcal{Y}+\xi^2_{\gamma}m_{1n}(z))^s\big)+\rO_{\prec}(n^{-4/3+3\epsilon})\nonumber \\
            &+F^{(2)}(n\eta_0\operatorname{Im}m_{*,\gamma})\times\frac{\phi^{-2}\eta^2_0}{2}\big[\operatorname{Im}\big(C_0\mathcal{X}\big)^2+2\operatorname{Im}\big(C_0\mathcal{X}\big)\times\operatorname{Im}\big(C_1\mathcal{X}(\mathcal{Y}+\xi^2_{\gamma}m_{1n}(z))\big)\big]+\rO_{\prec}(n^{-4/3+\epsilon}) \nonumber \\
            &+F^{(3)}(n\eta_0\operatorname{Im}m_{*,\gamma})\times \frac{\phi^{-3}\eta^3_0}{6}\big[\operatorname{Im}\big(C_0\mathcal{X}\big)^3+\rO_{\prec}(n^{-4/3+\epsilon})\big]+\rO_{\prec}(n^{-4/3+C\epsilon}). \label{eq_first order decomp}
\end{align}
Moreover, for some constants $C_{l,k}=\rO(1), 0 \leq l \leq 3$ and $0 \leq k \leq 2,$ using \eqref{eq_first order decomp}, we can further obtain 
\begin{gather*}
    \begin{split}
       &\mathbb{E}\big(F(n\eta_0\operatorname{Im}m_{p,\gamma-1})\big)-\mathbb{E}\big(F(n\eta_0\operatorname{Im}m_{*,\gamma})\big)\\
        &=\mathbb{E}\Big(F^{(1)}(n\eta_0\operatorname{Im}m_{*,\gamma})\times\phi^{-1}\eta_0\operatorname{Im}\big(\sum_{0\le k\le 2}C_{1,k}\mathcal{X}\mathcal{Y}^k\big)\Big)+\mathbb{E}\Big(F^{(2)}(n\eta_0\operatorname{Im}m_{*,\gamma})\times\phi^{-2}\eta^2_0\big(\frac{C_{2,0}}{2}(\operatorname{Im}\mathcal{X})^2+C_{2,1}\operatorname{Im}\mathcal{X}\operatorname{Im}\mathcal{X}\mathcal{Y}\big)\Big)\\
        &+\mathbb{E}\Big(F^{(3)}(n\eta_0\operatorname{Im}m_{*,\gamma})\times\phi^{-3}\eta^3_0\big(\frac{C_{3,0}}{6}(\operatorname{Im}\mathcal{X})^3\big))\Big)+\rO(n^{-4/3+C\epsilon}).
    \end{split}
\end{gather*}
We note that the above expression also holds for the second difference in \eqref{eq_greenfuncomp_diffdecomp2} as we replace $\mathcal{X}$, $\mathcal{Y}$ with $\mathcal{\widetilde{X}}$, $\mathcal{\widetilde{Y}}$. In light of (\ref{eq_greenfuncomp_diffdecomp2}), it suffices to consider the following difference for any $1\le\alpha\le3$, $1\le\alpha+\beta\le3$,
\begin{gather*}
    \mathbb{E}\Big(\eta_0^{\alpha}(\mathbf{y}_{\gamma}^{*}(G_{\gamma}^{(\gamma)})^2\mathbf{y}_{\gamma})^{\alpha}(\mathbf{y}_{\gamma}^{*}G_{\gamma}^{(\gamma)}\mathbf{y}_{\gamma})^{\beta}\Big)-\mathbb{E}\Big(\eta_0^{\alpha}(\mathbf{\widetilde{y}}_{\gamma}^{*}(G_{\gamma}^{(\gamma)})^2\mathbf{\widetilde{y}}_{\gamma})^{\alpha}(\mathbf{\widetilde{y}}_{\gamma}^{*}G_{\gamma}^{(\gamma)}\mathbf{\widetilde{y}}_{\gamma})^{\beta}\Big).
\end{gather*}
We only consider the case when $\gamma=1$, and other cases can be handled similarly. 

First, when $\alpha=1$, $\beta=0$, since $\mathbf{y}_1$ and {$\mathbf{\widetilde{y}}_1$} are independent with $G_1^{(1)}$, we conclude from (\ref{lem: basictool_momentsofu}) that $\mathbb{E}\eta_0\mathbf{y}_1^{*}(G_1^{(1)})^2\mathbf{y}_1-\mathbb{E}\eta_0\mathbf{\widetilde{y}}_1^{*}(G_1^{(1)})^2\mathbf{\widetilde{y}}_1=0.$ For $\alpha=3$, $\beta=0$, we have that 
\begin{gather*}
\begin{split}
    \mathbb{E}\Big(\eta_0^3\big(\mathbf{y}_1^{*}(G_1^{(1)})^2\mathbf{y}_1\big)^3\Big)-\mathbb{E}\Big(\eta_0^3\big(\mathbf{\widetilde{y}}_1^{*}(G_1^{(1)})^2\mathbf{\widetilde{y}}_1\big)^3\Big)
    \lesssim\eta_0^3\big(\mathcal{M}_2+\mathcal{M}^{1}_4+\mathcal{M}_4^2+\mathcal{M}_6\big).
\end{split}
\end{gather*}
where
\begin{gather*}
    \mathcal{M}_2:=\mathbb{E}\sum_{i,j,k}\big(u_{i1}^2u_{j1}^2u_{k1}^2[(G_1^{(1)})^2]_{ij}[(G_1^{(1)})^2]_{jk}[(G_1^{(1)})^2]_{ki}\big)-\mathbb{E}\sum_{i,j,k}\big(z_{i1}^2z_{j1}^2z_{k1}^2[(G_1^{(1)})^2]_{ij}[(G_1^{(1)})^2]_{jk}[(G_1^{(1)})^2]_{ki}\big),\\
    \mathcal{M}^1_4:=\mathbb{E}\sum_{i,j}\big(u_{i1}^4u_{j1}^2[(G_1^{(1)})^2]_{ij}^2[(G_1^{(1)})^2]_{ii}\big)-\mathbb{E}\sum_{i,j}\big(z_{i1}^4z_{j1}^2[(G_1^{(1)})^2]_{ij}^2[(G_1^{(1)})^2]_{ii}\big),\\
    \mathcal{M}^2_4:=\mathbb{E}\sum_{i,j}\big(u_{i1}^4u_{j1}^2[(G_1^{(1)})^2]_{ii}^2[(G_1^{(1)})^2]_{jj}\big)-\mathbb{E}\sum_{i,j}\big(z_{i1}^4z_{j1}^2[(G_1^{(1)})^2]_{ii}^2[(G_1^{(1)})^2]_{jj}\big),\\
    \mathcal{M}_6:=\mathbb{E}\sum_iu_{i1}^6[(G_1^{(1)})^2]_{ii}^3-\mathbb{E}\sum_iz_{i1}^6[(G_1^{(1)})^2]_{ii}^3.
\end{gather*}
It follows from \eqref{eq_greenfuncomp_priorest3} and \cite[Theorem V.1]{Wen2021} that $\mathcal{M}_2\lesssim p^{-1}n^{1+3\epsilon},\quad \mathcal{M}_4^{1,2}\lesssim p^{-1}n^{1+3\epsilon},\quad \mathcal{M}_6\lesssim p^{-2}n^{1+3\epsilon}.$ This implies that $  \mathbb{E}\Big(\eta_0^3\big(\mathbf{y}_1^{*}(G_1^{(1)})^2\mathbf{y}_1\big)^3\Big)-\mathbb{E}\Big(\eta_0^3\big(\mathbf{\widetilde{y}}_1^{*}(G_1^{(1)})^2\mathbf{\widetilde{y}}_1\big)^3\Big)\lesssim n^{-2+3\epsilon}.$

Second, for the other cases, we will need a finer investigation. We focus on the case $\alpha=1$, $\beta=2$ and the other cases can be handled similarly. Let $y_{i1}$, {$\widetilde{y}_{i1}$} be the $i$-th element of $\mathbf{y}_1$ and {$\mathbf{\widetilde{y}}_1$} respectively. We use the notation $\mathcal{E}(ab,ij,st):=\mathbb{E}[(G_1^{(1)})^2]_{ab}[G_1^{(1)}]_{ij}[G_1^{(1)}]_{st}$ in the sequel. We have 
\begin{gather}\label{eq_greenfuncomp_decomp3}
\begin{split}
    &\mathbb{E}\big(\eta_0\mathbf{y}_1^{*}(G_1^{(1)})^2\mathbf{y}_1(\mathbf{y}_1^{*}G_1^{(1)}\mathbf{y}_1)^2\big)=\eta_0\sum_{k_1,k_2,k_3}^{*}\mathbb{E}\prod_{i=1}^3y_{k_i1}^2\big(\mathcal{E}(k_1k_1,k_2k_2,k_3k_3)+2\mathcal{E}(k_1k_1,k_2k_3,k_2k_3)\\
    &+4\mathcal{E}(k_1k_2,k_1k_2,k_3k_3)+8\mathcal{E}(k_1k_2,k_1k_3,k_2k_3)\big)+\eta_0\sum_{\mathcal{J}_6}\mathbb{E}\prod_{i=1}^6y_{k_i,1}\mathcal{E}(k_1k_2,k_3k_4,k_5k_6),
\end{split}
\end{gather}
where $\sum^{*}_{k_1,\dots,k_l}$ denotes the sum over $\{(k_1,\dots,k_l)\in\{1,\dots,p\}^l:\text{$k_i$'s are distinct}\}$ and $\mathcal{J}_6$ denotes the set of indices $k_i\in\{1,\dots,p\},i=1,\dots, 6$ such that $k_i$ appears even number of times and there is an index $k_i$ appears at least four times. Note that the odd number of an index $k_i$ will give the degenerated expectation. This gives the estimate $\#\mathcal{J}_6\lesssim n^2$, together with \eqref{eq_greenfuncomp_priorest2} and \eqref{eq_greenfuncomp_priorest3}, we have $|\eta_0\sum_{\mathcal{J}_6}\mathbb{E}\prod_{i=1}^6y_{k_i1}\mathcal{E}(k_1k_2,k_3k_4,k_5k_6)|\lesssim {n^{-4/3+C\epsilon}}.$ For other terms in \eqref{eq_greenfuncomp_decomp3}, we take difference with their counterparts in {$\mathbb{E}(\eta_0\mathbf{\widetilde{y}}_1^{*}(G_1^{(1)})^2\mathbf{\widetilde{y}}_1(\mathbf{\widetilde{y}}_1^{*}G_1^{(1)}\mathbf{\widetilde{y}}_1)^2)$ 
\begin{gather*}
    \mathbb{E}\big(\eta_0\mathbf{y}_1^{*}(G_1^{(1)})^2\mathbf{y}_1(\mathbf{y}_1^{*}G_1^{(1)}\mathbf{y}_1)^2\big)-\mathbb{E}\big(\eta_0\mathbf{\widetilde{y}}_1^{*}(G_1^{(1)})^2\mathbf{\widetilde{y}}_1(\mathbf{\widetilde{y}}_1^{*}G_1^{(1)}\mathbf{\widetilde{y}}_1)^2\big)=:\sum_{i=1}^42^{i-1}\Delta_i+\rO(n^{-4/3+C\epsilon}),
\end{gather*}}
where 
\begin{gather*}
    \Delta_1:=\eta_0\sum_{k_1,k_2,k_3}^{*}\Delta_y\mathcal{E}(k_1k_1,k_2k_2,k_3k_3),\quad \Delta_2:=\eta_0\sum_{k_1,k_2,k_3}^{*}\Delta_y\mathcal{E}(k_1k_1,k_2k_3,k_2k_3),\\
    \Delta_3:=\eta_0\sum_{k_1,k_2,k_3}^{*}\Delta_y\mathcal{E}(k_1k_2,k_1k_2,k_3k_3),\quad \Delta_4:=\eta_0\sum_{k_1,k_2,k_3}^{*}\Delta_y\mathcal{E}(k_1k_2,k_1k_3,k_2k_3),
\end{gather*}
with {$\Delta_y:=\mathbb{E}(\prod_{i=1}^3y_{k_i1}^2-\prod_{i=1}^3\widetilde{y}_{k_i1}^2)$.} In the following, we aim to show that {$\Delta_i=\rO(n^{-1-c+C\epsilon})$} for each $i$. In fact, one may see that
{
\begin{gather*}
\begin{split}
    \Delta_y=\mathbb{E}\xi^6_1\big(u_{k_11}^2u_{k_21}^2u_{k_31}^2-z_{k_11}^2z_{k_21}^2z_{k_31}^2-3u_{k_11}^2u_{k_21}^2z_{k_31}^2+3u_{k_11}^2z_{k_21}^2z_{k_31}^2\big)=\rO(n^{-4}),
\end{split}
\end{gather*}}
where we again used \cite[Theorem V.1]{Wen2021} and the fact that $k_1,k_2,k_3$ are distinct. Together with Lemma \ref{lem_somekeyestimates}, we conclude that for each $i=1,\dots,4$, $\Delta_i={\rO(n^{-1-c+C\epsilon})}.$
From above estimates, one can easily obtain that
\begin{gather*}
    \mathbb{E}\big(\eta_0\mathbf{y}_1^{*}(G_1^{(1)})^2\mathbf{y}_1(\mathbf{y}_1^{*}G_1^{(1)}\mathbf{y}_1)^2\big)-\mathbb{E}\big(\eta_0\mathbf{\widetilde{y}}_1^{*}(G_1^{(1)})^2\mathbf{\widetilde{y}}_1(\mathbf{\widetilde{y}}_1^{*}G_1^{(1)}\mathbf{\widetilde{y}}_1)^2\big)=\rO(n^{-1-c+C\epsilon}).
\end{gather*}
Combining the results of all cases of $\alpha$ and $\beta$, we can show \eqref{eq_greenfuncomp_diffdecomp2} and conclude the proof.
\end{proof}

\subsection{Proof of Lemma \ref{lem_somekeyestimates}}\label{sec_append_prior}
In this subsection, we show the estimates in \eqref{eq_greenfuncomp_priorest1}-\eqref{eq_greenfuncomp_priorest3} following Appendix B of \cite{bao2022extreme}.
\begin{proof}[\bf Proof of \eqref{eq_greenfuncomp_priorest1}]
    Recall that $z=E+\rm i\eta_0$ with $\eta_0=n^{-2/3-\epsilon}$. By Cauchy's integral formula, we have
    \begin{gather*}
        \mathbf{y}_{\gamma}^{*}(G_{\gamma}^{(\gamma)})^2\mathbf{y}_{\gamma}=\frac{1}{2\pi \rm i}\oint_{\Gamma_{\eta_0}}\frac{\mathbf{y}_{\gamma}^{*}G_{\gamma}^{(\gamma)}(x)\mathbf{y}_{\gamma}}{(x-z)^2}\mathrm{d} x,
    \end{gather*}
    where $\Gamma_{\eta_0}$ is the disk centered at $z$ with radius $\eta_0/2$. Decompose that 
    \begin{gather*}
        |\mathbf{y}_{\gamma}^{*}(G_{\gamma}^{(\gamma)})^2\mathbf{y}_{\gamma}|\lesssim \left|\oint_{\Gamma_{\eta_0}}\frac{\xi^2_{\gamma}\operatorname{tr}G_{\gamma}^{(\gamma)}(x)}{p(x-z)^2}\mathrm{dx}\right|+\left|\oint_{\Gamma_{\eta_0}}\frac{\mathbf{y}_{\gamma}^{*}G_{\gamma}^{(\gamma)}(x)\mathbf{y}_{\gamma}-\xi^2_{\gamma}p^{-1}\operatorname{tr}G_{\gamma}^{(\gamma)}(x)}{(x-z)^2}\mathrm{d} x\right|.
    \end{gather*}
For the first term, we see from Cauchy's integral formula, Theorem \ref{thm_main_locallaws}, Lemmas \ref{lem: basictool_greenfunciden}, \ref{lem: basictool_largedevia} and \ref{thm_main_asymptotic laws} that   
    \begin{gather*}
        \left|\oint_{\Gamma_{\eta_0}}\frac{\xi^2_{\gamma}\operatorname{tr}G_{\gamma}^{(\gamma)}(x)}{p(\mathrm{x}-z)^2}\mathrm{dx}\Big|\lesssim\Big|\frac{\xi^2_{\gamma}\operatorname{tr}(G_{\gamma}^{(\gamma)}(z))^2}{p}\right|\lesssim\frac{1}{p}\sum_{i=1}^p\frac{1}{|\lambda_i^{(\gamma)}-z|^2}=\frac{\operatorname{Im}\operatorname{tr}\mathcal{G}_{\gamma}^{(\gamma)}}{p\eta_0}\prec\frac{1}{\sqrt{\eta_0}}=n^{1/3+\epsilon}.
        \end{gather*}
With similar reasoning, for the second term, we have that 
    \begin{gather*}
        \left|\oint_{\Gamma_{\eta_0}}\frac{\mathbf{y}_{\gamma}^{*}G_{\gamma}^{(\gamma)}(x)\mathbf{y}_{\gamma}-\xi^2_{\gamma}p^{-1}\operatorname{tr}G_{\gamma}^{(\gamma)}(x)}{(x-z)^2}\mathrm{dx}\right|\prec\Big(\frac{\operatorname{Im}\operatorname{tr}G_{\gamma}^{(\gamma)}}{p^2\eta_0^3}\Big)^{1/2}\lesssim\Big(\frac{\operatorname{Im}\operatorname{tr}\mathcal{G}_{\gamma}^{(\gamma)}}{p^2\eta_0^3}\Big)^{1/2}\prec n^{1/3+\epsilon}.
    \end{gather*}
This completes our proof. 
\end{proof}
\begin{proof}[\bf Proof of \eqref{eq_greenfuncomp_priorest2}]
 For the diagonal case,  we only show the result for {$(G_0(z))_{11}$} when $z=E+\mathrm{i}\eta_0$ with $\eta_0=n^{-2/3-\epsilon}$, $|E-\lambda_{+}|\le n^{-2/3+\epsilon}$ for $\gamma=0$, while other cases $(G_{\gamma}^{(\gamma)}(z))_{ii}, \gamma \geq 1$ can be handled similarly. By resolvent identity (see equation (5.5) of \cite{yang2019edge}), one has
    \begin{gather*}
        (G_0(z))_{11}=\frac{1}{-z(1+\mathbf{r}_1^{*} \mathcal{G}_0^{[1]}(z)\mathbf{r}_1)}=\frac{1}{-z\big(1+\sigma_1(m_{2n}^{[1]}(z)+\mathrm{R}_1+\mathrm{R}_2+\mathrm{R}_3)\big)},
    \end{gather*}
    where $\mathbf{r}_i:=(\xi_1w_{i1}\sqrt{\sigma_i},\xi_2w_{i2}\sqrt{\sigma_i},\dots,\xi_nw_{in}\sqrt{\sigma_i})$ is $i$-th row of $V$ as in (\ref{eq_samplecov_t}) and {the upper index $[i]$ denotes the $i$-th row of the corresponding data matrix $Y$ being deleted, and} $\mathrm{R}_3:={\sum_{i\neq j}\xi_i\xi_jw_{1i}w_{1j}(\mathcal{G}_0^{[1]}(z) )_{ij}}$ and
    \begin{gather*}
        \mathrm{R}_1:={-z^{-1}\sum_i\xi^2_iw_{1i}^2(1+m_{1n}^{[1]}(z)D^2)^{-1}_{ii}}-m_{2n}^{[1]}(z),\quad \mathrm{R}_2:={\sum_i\xi^2_iw_{1i}^2\big((  \mathcal{G}_0^{[1]}(z) )_{ii}}+z^{-1}(1+m_{1n}^{[1]}(z)D^2)^{-1}_{ii}\big).
    \end{gather*}
It is easy to see that Theorem \ref{thm_main_locallaws} still applies to $\mathcal{G}_0^{[1]}(z)$ so that 
    \begin{gather*}
        \Big|( {\mathcal{G}_0^{[1]}(z)} )_{ij}+z^{-1}(1+m_{1n}^{[1]}(z)D^2)^{-1}_{ij}\Big|\prec\sqrt{\frac{\operatorname{Im} m_{1n}^{[1]}(z)}{ p \eta}}+\frac{1}{p \eta}.
    \end{gather*}
For $\mathrm{R}_3$, let $\{\varepsilon_i\}$ be i.i.d. Rademacher random variables, we see that $\mathrm{R}_3\overset{d}{=}\sum_{i\neq j}\xi_i\xi_j\varepsilon_i\varepsilon_j{ w_{1i}w_{1j}}( {\mathcal{G}_0^{[1]}(z)} )_{ij}.$
    Therefore, by large deviation of the quadratic form of $\varepsilon_i$'s, we have
    \begin{gather}\label{eq_argumentkey}
        \left|\sum_{i\neq j}\xi_i\xi_j\varepsilon_i\varepsilon_j{w_{1i}w_{1j}}({\mathcal{G}_0^{[1]}(z)} )_{ij}\right|\prec\Big(\sum_{i\neq j}\xi^2_i\xi^2_j{ w^2_{1i}w^2_{1j}}\big(( { \mathcal{G}_0^{[1]}(z)} )_{ij}\big)^2\Big)^{1/2}\prec p^{-1/3+\epsilon}.
    \end{gather}
    It follows that $\mathrm{R}_3\prec p^{-1/3+\epsilon}$. Next, for $\mathrm{R}_2$, we have
    \begin{gather*}
        |\mathrm{R}_2|\le\sum_i\xi^2_i{w_{1i}^2}\left|{ (\mathcal{G}_0^{[1]}(z)} )_{ii}+z^{-1}(1+m_{1n}^{[1]}(z)D^2)^{-1}_{ii}\right|\prec p^{-1/3+\epsilon}\sum_i\xi^2_i{w_{1i}^2}\prec n^{-1/3+\epsilon}.
    \end{gather*}
    Lastly, for $\mathrm{R}_1$, using the fact $|m_{2n}^{[1]}(z)|\sim1$, we have the trivial bound $|\mathrm{R}_1|\prec n^{-1/2}$ since {$w_{1i}^2$'s} are i.i.d.. Then by the definition of $m_{2n}^{[1]}$ and Lemma \ref{lem: basictool_greenfunciden}, together with above estimates, we have $(G_0(z))_{11}=-z^{-1}(1+\sigma_1m_{2n}(z))^{-1}+\rO_{\prec}(n^{-1/3+\epsilon}).$
    For $i\neq j$, by the resolvent identity (see equation (5.6) of \cite{yang2019edge}), one has
    \begin{gather*}
        |(G_0(z))_{ij}|=|z(G_0(z))_{ii}(G^{[1]}_{0}(z))_{jj}\mathbf{r}_i^{*}\mathcal{G}^{[ij]}_0(z)\mathbf{r}_j|\prec |\mathbf{r}_i^{*}\mathcal{G}^{[ij]}(z)\mathbf{r}_j|=\left|\sum_{k,l}\xi_k\xi_lw_{ik}w_{jl}\sqrt{\sigma}_k\sqrt{\sigma}_l(\mathcal{G}_0^{[ij]}(z))_{kl}\right|.
    \end{gather*}
Then we can conclude the proof using a discussion similar to (\ref{eq_argumentkey}) with the help of Theorem \ref{thm_main_locallaws}. 
\end{proof}
\begin{proof}[\bf Proof of \eqref{eq_greenfuncomp_priorest3}]
As in the proof of (\ref{eq_greenfuncomp_priorest2}), we only consider $\gamma=0.$ By Cauchy's integral formula, we have
    \begin{gather*}
        |((G_0(z))^2)_{ij}|\lesssim\Big|\oint_{\Gamma_{\eta_0}}\frac{(G_0(x))_{ij}}{(x-z)^2}\mathrm{d}x\Big|.
    \end{gather*}
    For the case $i\neq j$, we obtain straight forwards from \eqref{eq_greenfuncomp_priorest2} that
    \begin{gather*}
        \Big|\oint_{\Gamma_{\eta_0}}\frac{(G_0(x))_{ij}}{(x-z)^2}\mathrm{d}x\Big|\prec n^{1/3+\epsilon}.
    \end{gather*}
    If $i=j$,  we see from Lemma \ref{thm_main_asymptotic laws} that $|m_{2n}^{\prime}(z)|\lesssim n^{1/3+\epsilon}$. Then we have
    \begin{gather*}
    \begin{split}
        &\left|\oint_{\Gamma_{\eta_0}}\frac{ {(G_0(x))_{ii}} }{(x-z)^2}\mathrm{d}x\right|\lesssim\left|\oint_{\Gamma_{\eta_0}}\frac{ {(G_0(x))_{ii} +{x^{-1}(1+\sigma_i m_{2n}(x))^{-1}}}}{(x-z)^2}\mathrm{d}x \right|+\left|\oint_{\Gamma_{\eta_0}}\frac{{{x^{-1}(1+\sigma_i m_{2n}(x))^{-1}}}
        }{(x-z)^2}\mathrm{d}x \right|\\
        &\prec\frac{n^{-1/3+\epsilon}}{\eta_0}+\rO(m_{2n}^{\prime}(z))\prec n^{1/3+2\epsilon}.
    \end{split}
    \end{gather*}
    This completes the proof of \eqref{eq_greenfuncomp_priorest3}.
\end{proof}

\appendix

\section{Proof of the local laws}\label{sec_prooflocalaw}
We prove Theorem \ref{thm_main_locallaws} following \cite{yang2019edge}. Due to similarity, we only provide the key ingredients of the counterparts of \cite{yang2019edge}. 
\subsection{Basic tools and auxiliary lemmas}
In this section we collect some necessary notations and technical tools. Denote the singular value decomposition of $V$ in (\ref{eq_samplecov_t}) as $ V=\sum_{k=1}^{p\wedge n}\sqrt{\lambda_k}\gamma_k\zeta_k^{*}.$  Then for the resolvents in \eqref{eq_resolvents}, we can get for $1 \leq i,j \leq p$ and $1 \leq \mu,\nu \leq n$,
\begin{gather}\label{eq: spectraldecomp}
    G_{ij}=\sum_{k=1}^p\frac{z\gamma_k(i)\gamma^{*}_k(j)}{\lambda_k-z},\quad \mathcal{G}_{ji}=\sum_{k=1}^n\frac{\zeta_k(j)\zeta^{*}_k(i)}{\lambda_k-z}.
\end{gather}

First, we will frequently use the following identities whose proof can be found in \cite{DingandYang2018,PillaiandYin2014,yang2019edge}.  
\begin{lemma}\label{lem: basictool_resolvent}
For any $\mathcal{T}\subset \{1,2,\cdots,n\}$, we have that 
		\begin{eqnarray*}
			\mathcal{G}_{ii}^{(\mathcal{T})}(z) & = & -\frac{1}{z+z\mathbf{v}_{i}^{*}{ G}^{(i\mathcal{T})}(z)\mathbf{v}_{i}},\qquad\forall i\in\mathcal{I}\setminus \mathcal{T},\\
			\mathcal{G}_{ij}^{(\mathcal{T})}(z) & = & z\mathcal{G}_{ii}^{(\mathcal{T})}(z)\mathcal{G}_{jj}^{(i\mathcal{T})}(z)\mathbf{v}_{i}^{*}{ G}^{(ij\mathcal{T})}(z)\mathbf{v}_{j},\qquad\forall i,j\in\mathcal{I}\setminus \mathcal{T},i\ne j,\\
			\mathcal{G}_{ij}^{(\mathcal{T})}(z) & = & \mathcal{G}_{ij}^{(k\mathcal{T})}(z)+\frac{\mathcal{G}_{ik}^{(\mathcal{T})}(z)\mathcal{G}_{kj}^{(\mathcal{T})}(z)}{\mathcal{G}_{kk}^{(\mathcal{T})}(z)},\qquad\forall i,j,k\in\mathcal{I}\setminus \mathcal{T},i,j\ne k.
		\end{eqnarray*}
Moreover,
\begin{gather*}
    \sum_{1\le i\le p}|G_{ji}|^2=\sum_{1\le i\le p}|G_{ij}|^2=\frac{|z|^2}{\eta}\operatorname{Im}(\frac{G_{jj}}{z}), \  \sum_{1\le\mu\le n}|\mathcal{G}_{\nu\mu}|^2=\sum_{1\le\mu\le n}|\mathcal{G}_{\mu\nu}|^2=\frac{\operatorname{Im}\mathcal{G}_{\nu\nu}}{\eta},\\
    \|G\Sigma\|_F^2=\eta^{-1}\operatorname{Im}\operatorname{tr}(G\Sigma^2).
\end{gather*} 		
\end{lemma}
Second, the following estimates are also for our discussions whose proof can also be found in \cite{DingandYang2018,PillaiandYin2014,yang2019edge}.  

\begin{lemma}\label{lem: basictool_greenfunciden}
 The following estimates hold uniformly for all $z\in\mathbf{D}$ in (\ref{eq: spectraldomainD}) and $C>0$,
\begin{gather*}
    \|G\|+\|\mathcal{G}\|\le C\eta^{-1},\quad \|\partial_zG\|+\|\partial_z\mathcal{G}\|\le C\eta^{-2}.
\end{gather*}
The above estimates remain true for $G^{(\mathcal{T})}$ instead of $G$ for any $\mathcal{T}\subset\{1,\dots, p\}$
  \begin{gather*}
        |m_1-m_1^{(\mathcal{T})}|+|m_2-m_2^{(\mathcal{T})}|\le\frac{C|\mathcal{T}|}{n\eta}.
    \end{gather*}
\end{lemma}
Finally, the following estimates will be used in our calculations whose proof can be found in \cite{DingandYang2018,Wen2021}. 

\begin{lemma}\label{lem: basictool_largedevia}
(1).     Let $\mathbf{u}=(u_1,\dots,u_p)$ be a $p$-dimensional random vector of spherical uniform distribution. For any $i\neq j\in\{1,\dots,p\}$, let $k_1,k_2$ be two non-negative integers, we have
    \begin{gather}\label{lem: basictool_momentsofu}
        \mathbb{E}|u_i^2|=\frac{1}{p},\quad \mathbb{E}|u_i^4|=\frac{3}{p(2+p)},\quad \mathbb{E}|u_i^2u_j^2|=\frac{1}{p(2+p)},\quad
        \mathbb{E}|u_i^{k_1}u_j^{k_2}|=0,\quad \text{if $k_1+k_2$ is odd.}
    \end{gather}
(2). $\mathbf{w}=(w_{1},\dots,w_{p})^{*}$,
		$\widetilde{\mathbf{w}}=(\widetilde{w}_{1},\dots,\widetilde{w}_{p})^{*}$ be two independent random vectors as in (\ref{eq_samplecov_t}). Suppose $A=(a_{ij})$  an $p\times p$ matrix and $\mathbf{b}=(b_{1},\dots,b_{p})^{*}$
		an $p$-dimensional vector, where $A$ and $\mathbf{b}$ may be
		complex-valued and $\mathbf{w},\tilde{\mathbf{w}},A,\mathbf{b}$ are
		independent. Then as $p\to\infty$
		\begin{eqnarray*}
		|\mathbf{b}^{*}\mathbf{w}|  \prec \sqrt{\frac{\|\mathbf{b}\|^{2}}{p},} \		|\mathbf{w}^{*}A\mathbf{w}-\frac{1}{p}{\rm tr}A|  \prec  \frac{1}{p}\|A\|_{F}, \	\Big|\mathbf{w}^{*}A\widetilde{\mathbf{w}}\Big| \prec  \frac{1}{p}\|A\|_{F}.
		\end{eqnarray*}
\end{lemma}

\subsection{Proof of Theorem \ref{thm_main_locallaws}}
For simplicity, following \cite{Ding2023,ding2021spiked,yang2019edge} we introduce the following notations. 
  \begin{gather*}
      \Pi_1(z):=-(1+m_{2n}(z)\Sigma)^{-1},\quad \Pi_2(z):=-z^{-1}(1+m_{1n}(z)D^2)^{-1}.
  \end{gather*}
Note that $ \frac{1}{zp}\operatorname{tr}\Pi_1=m_n(z).$ Moreover, we denote $m_2(z):=p^{-1}\sum_{i=1}^n\xi^2_i\mathcal{G}_{ii}(z).$ In addition, it is more convenient to introduce the following parameters
\begin{gather*}
\Lambda\equiv\Lambda(z):=\max_{i,j}|\mathcal{G}-\Pi_2|_{ij},\quad \Theta\equiv\Theta(z):=|m_1(z)-m_{1n}(z)|+|m_2(z)-m_{2n}(z)|,\\
  \Lambda_o\equiv\Lambda_o(z):=\max_{i\neq j}|\mathcal{G}_{ij}|,\quad  \Psi_{\Theta}\equiv\Psi_{\Theta}(z):=\sqrt{\frac{\operatorname{Im}m_{1n}(z)+\Theta}{p\eta}}+\frac{1}{p\eta}.
\end{gather*}
It is easy to see that on the event $\Xi:=\{\Lambda\le(\log p)^{-1}\},$ the following holds uniformly on $i,j\in \{1,2,\cdots,n\}$ and $z\in\mathbf{D}$
\begin{gather}\label{lem: basictool_prior}
    \{\mathbf{1}(\Xi)+\mathbf{1}(\eta\ge1)\}|\mathcal{G}_{ij}^{(\mathcal{T)}}|+\mathbf{1}(\Xi)|(\mathcal{G}^{(\mathcal{T})}_{ii})^{-1}|=\rO_{\prec}(1).
\end{gather}
Moreover, by \eqref{eq: asymlaws_main_estm12n} and \eqref{eq: asymlaws_main_lowerbound}, when $\Theta\prec(n\eta)^{-1}$,
\begin{gather*}
    \|\Pi_{1,2}\|=\rO(1),\quad \Psi_{\Theta}\gtrsim p^{-1/2},\quad \Psi_{\Theta}^2\lesssim(p\eta)^{-1},\quad \Psi_{\Theta}(z)\sim\sqrt{\frac{\operatorname{Im}m_{1n}(z)}{p\eta}}+\frac{1}{p\eta},\quad z\in\mathbf{D}_0.
\end{gather*}
Furthermore, we introduce the $Z$ variable
\begin{gather*}
    Z_{i}^{(\mathcal{T})}:=(1-\mathbb{E}_i)\big[\mathbf{v}_i^{*}G^{(i\mathcal{T})}\mathbf{v}_i\big],\quad i\notin\mathcal{T},
\end{gather*}
where $\mathbb{E}_i[\cdot]:=\mathbb{E}[\cdot|W^{(i)}]$ is the partial expectation over the randomness of the $i$-th row and column of $W$ in (\ref{eq_samplecov_t}). Thanks to (\ref{lem: basictool_momentsofu}), we see that observe that for both choice of $\mathbf{w}_i$, it follows
\begin{gather*}
    Z_i=\mathbf{v}_i^{*}G^{(i)}\mathbf{v}_i-\frac{\xi^2_i}{p}\operatorname{tr}(G^{(i)}\Sigma).
\end{gather*}

The following lemma is a counterpart of Lemma 5.9 of \cite{yang2019edge}. 
\begin{lemma}\label{lem: entrywiselocallaw_lem1}
    Suppose assumptions of Theorem \ref{thm_main_locallaws} hold.  Then for all $1 \leq i \leq n$ and $z \in \mathbf{D}$ uniformly
    \begin{gather*}
        \{\mathbf{1}(\Xi)+\mathbf{1}(\eta\ge1)\}(|Z_i|+\Lambda_o)\prec\Psi_{\Theta}.
    \end{gather*}
\end{lemma}
\begin{proof}
Let $\mathcal{I}=\{1,2,3,\cdots,n\}.$ Applying Lemmas \ref{lem: basictool_resolvent}, \ref{lem: basictool_largedevia} and \ref{thm_main_asymptotic laws}, we obtain that uniformly for $z\in\mathbf{D}$ and $i,j\in\mathcal{I}$ with $i\neq j$,
    \begin{gather}\label{eq: entrywise_eq1}
        \mathbf{1}(\Xi)|\mathcal{G}_{ij}|\le\mathbf{1}(\Xi)|z||\xi_i\xi_j||\mathcal{G}_{ii}\mathcal{G}_{jj}^{(i)}||\mathbf{v}_i^{*}G^{(ij)}\mathbf{v}_j|\prec\mathbf{1}(\Xi)|\mathcal{G}_{ii}\mathcal{G}_{jj}^{(i)}||\xi_i\xi_j|\frac{1}{p}\|\Sigma^{1/2}\|\|G^{(ij)}\Sigma^{1/2}\|_F.
    \end{gather}
Together with Lemmas \ref{lem: basictool_resolvent} and \ref{lem: basictool_greenfunciden}, we can see that    
\begin{gather*}
    \mathbf{1}(\Xi)|\mathcal{G}_{ij}|\prec\sqrt{\frac{\operatorname{Im}m_{1}^{(ij)}}{p\eta}}\prec\sqrt{\frac{\operatorname{Im}m_{1n}+\Theta}{p\eta}}+\frac{1}{p\eta}.
\end{gather*}
Therefore, $\mathbf{1}(\Xi)\Lambda_o\prec\Psi_{\Theta}$. The result for $Z_i$ under $\Xi$ can be derived similarly  in the sense that 
\begin{gather*}
    \mathbf{1}(\Xi)Z_i=\mathbf{1}(\Xi)\big(\mathbf{v}_i^{*}G^{(i)}\mathbf{v}_i-\frac{\xi^2_i}{p}\operatorname{tr}(G^{(i)}\Sigma)\big)\prec\mathbf{1}(\Xi)\frac{\xi^2_i}{p}\|\Sigma^{1/2}\|\|G^{(i)}\Sigma^{1/2}\|_F \prec \Psi_{\Theta}.
\end{gather*}
Similar procedure can be applied to $\mathcal{G}_{ij}$ and $Z_i$ via (\ref{lem: basictool_prior}). This completes the proof. 
\end{proof}

The following lemma is a counterpart of Lemma 5.10 of \cite{yang2019edge}. 
\begin{lemma}\label{lem: entrywiselocallaw_lem2}
    Suppose assumptions in Theorem \ref{thm_main_locallaws} hold. Then
    \begin{align}
        \mathbf{1}(\eta\ge 1)|F_p(m_1(z),z)|\prec p^{-1/2},\quad \mathbf{1}(\eta\ge1)\left|m_2(z)+\frac{1}{p}\sum_{i=1}^n\frac{\xi^2_i}{z(1+\xi^2_im_{1}(z))}\right|\prec p^{-1/2}, \label{eq_mmmmmmmm1} \\
        \mathbf{1}(\Xi)|F_p(m_1(z),z)|\prec \Psi_{\Theta},\quad \mathbf{1}(\Xi)\left|m_2(z)+\frac{1}{p}\sum_{i=1}^n\frac{\xi^2_i}{z(1+\xi^2_im_{1}(z))}\right|\prec \Psi_{\Theta}. \label{eq_mmmmmmmm}
    \end{align}
    Moreover, we have the finer estimates
    \begin{gather}\label{eq: locallaw_finerest_F}
        \mathbf{1}(\Xi)|F_p(m_1(z),z)|\prec\mathbf{1}(\Xi)\Big(|\mathbf{Z}_1|+|\mathbf{Z}_2|\Big)+\Psi_{\Theta}^2,
    \end{gather}
    and 
    \begin{gather}\label{eq: locallaw_finerest_m}
       \mathbf{1}(\Xi) |m_2(z)+\frac{1}{p}\sum_{i=1}^n\frac{\xi^2_i}{z(1+\xi^2_im_{1}(z))}|\prec\mathbf{1}(\Xi)\big(|\mathbf{Z}_2|\big)+\Psi_{\Theta}^2,
    \end{gather}
    where 
    \begin{gather*}
        \mathbf{Z}_1:=\frac{1}{p}\sum_{j=1}^n\frac{\xi^2_j}{z(1+\xi^2_jm^{(j)}_1(z))}Z_j,\quad \mathbf{Z}_2:=\frac{1}{p}\sum_{j=1}^n\frac{\xi^2_j}{z(1+\xi^2_jm^{(j)}_1(z))^2}Z_j.
    \end{gather*}
\end{lemma}
\begin{proof}
    We begin the proof with the finer estimates (\ref{eq: locallaw_finerest_F}) and (\ref{eq: locallaw_finerest_m}). By Lemmas \ref{lem: basictool_resolvent}, \ref{lem: basictool_greenfunciden}, \ref{lem: basictool_largedevia} and \ref{thm_main_asymptotic laws}
    \begin{gather}\label{eq: expan_cal G}
        \mathbf{1}(\Xi)\mathcal{G}_{jj}=\mathbf{1}(\Xi)\Big(-\frac{1}{z(1+\xi_j^2m_1(z))}+\frac{Z_j}{z(1+\xi^2_jm_1^{(j)}(z))^2}+\rO_{\prec}(\Psi_{\Theta}^2)\Big),
    \end{gather}
which results in (\ref{eq: locallaw_finerest_m}) by taking average with $p^{-1}\sum_j\xi^2_j$. Moreover, using (B.14) of \cite{Ding2023}, we see that 
    \begin{gather*}
    \begin{split}
       G&=-z^{-1}(I+m_2(z)\Sigma)^{-1}+z^{-1}\sum_{i=1}^n\frac{G^{(i)}(\mathbf{v}_i\mathbf{v}_i^{*}-p^{-1}\xi^2_i\Sigma)}{1+\mathbf{v}_i^{*}G^{(i)}\mathbf{v}_i}(I+m_2(z)\Sigma)^{-1}+z^{-1}\frac{1}{p}\sum_{i=1}^n\frac{(G^{(i)}-G)\xi^2_i\Sigma}{1+\mathbf{v}_i^{*}G^{(i)}\mathbf{v}_i}(I+m_2(z)\Sigma)^{-1}\\
    &:=-z^{-1}(I+m_2(z)\Sigma)^{-1}+R_1+R_2.
    \end{split}
    \end{gather*}
    Taking average with $p^{-1}\sum_i\sigma_i$, we have
     \begin{gather}\label{eq_m1}
         m_1(z)=\frac{1}{p}\operatorname{tr}(G(z)\Sigma)=-\frac{1}{p}\sum_{i=1}^p\frac{\sigma_i}{z(1+m_2(z)\sigma_i)}+\frac{1}{p}\operatorname{tr}(R_1\Sigma)+\frac{1}{p}\operatorname{tr}(R_2\Sigma).
     \end{gather}
Similar to (\ref{eq: expan_cal G}), we have that
     \begin{gather*}
         \mathbf{1}(\Xi)\Big(\frac{1}{p}\operatorname{tr}(R_1\Sigma)\Big)=\mathbf{1}(\Xi)\rO_{\prec}(\mathbf{Z}_1+\Psi^2_{\Theta}),\quad
         \mathbf{1}(\Xi)\Big(\frac{1}{p}\operatorname{tr}(R_2\Sigma)\Big)=\mathbf{1}(\Xi)\rO_{\prec}(\mathbf{Z}_2+\Psi^2_{\Theta}).
     \end{gather*}
This results in (\ref{eq: locallaw_finerest_F}) using (\ref{eq:F(m,z)}). 

Armed with the above results, (\ref{eq_mmmmmmmm}) follows from       Lemma \ref{lem: entrywiselocallaw_lem1}. For (\ref{eq_mmmmmmmm1}), the arguments are similar except we need to prove that 
in the case $\eta\ge1$, with high probability, for all $1 \leq i \leq p, 1 \leq j \leq n$ and some constant $c'>0$
     \begin{gather}\label{eq: lowerbound_m12_eta>1}
         |1+\xi^2_jm_1|\ge c^{\prime},\quad |1+\sigma_im_2|\ge c^{\prime}.
     \end{gather}
The proof of (\ref{eq: lowerbound_m12_eta>1}) follows lines between equations (5.43) and (5.45) of \cite{yang2019edge} and we omit the details. This concludes our proof.           
\end{proof}

The following lemma is a counterpart of Lemma 5.11 of \cite{yang2019edge}. 
\begin{lemma}\label{lem: stability of F}
    Suppose assumptions in Theorem \ref{thm_main_locallaws} hold. Suppose a $z$-dependent function $\delta$ satisfying $p^{-1}\le\delta(z)\le\log^{-1}p$ for $z\in\mathbf{D}$ and assume that $\delta(z)$ is Lipschitz continuous with Lipschitz constant $\le p^{2}$. Suppose moreover that for each fixed $E$, the function $z\mapsto\delta(E+\mathrm{i}\eta)$ is non-increasing for $\eta>0$. Suppose that $\mu_0:\mathbf{D}\rightarrow\mathbb{C}$ is the Stieltjes transform of a probability measure. Let $z\in\mathbf{D}$ and suppose that for all $z^{\prime}\in\operatorname{Lip}(z)$, we have $|F_p(\mu_0,z)|\le\delta(z^{\prime}),$ then we have that for some constant $C>0$ 
    \begin{gather*}
        |\mu_0-m_{1n}(z)|\le\frac{C\delta}{\sqrt{\kappa+\eta+\delta}}.
    \end{gather*}
\end{lemma}
\begin{proof}
    The proof of this lemma of relies on the first and second statements in Lemma \ref{thm_main_asymptotic laws}. One can refer to \cite[Lemma 5.11]{yang2019edge} for more details.
\end{proof}

With Lemma \ref{lem: stability of F}, we see from \eqref{eq_mmmmmmmm1}, (\ref{eq_mmmmmmmm}) and \eqref{eq: lowerbound_m12_eta>1} that 
\begin{gather}\label{eq: est_theta_eta>1}
    \mathbf{1}(\eta\ge1)\Theta(z)\prec p^{-1/2}.
\end{gather}
Then from Lemma \ref{lem: entrywiselocallaw_lem1}, we have for the off-diagonal entries, $\mathbf{1}(\eta\ge1)\Lambda_o\prec p^{-1/2}.$
For the diagonal entries, using Lemma \ref{lem: entrywiselocallaw_lem1}, \eqref{eq: est_theta_eta>1} and \eqref{eq: lowerbound_m12_eta>1} with the following expression
\begin{gather*}
    \mathbf{1}(\eta\ge1)\mathcal{G}_{ii}=\mathbf{1}(\eta\ge1)\Big(-\frac{1}{z(1+\xi^2_im_1)}+\frac{Z_i}{z(1+\xi^2_im_1)^2}+\rO_{\prec}(\Psi_{\Theta}^2)\Big),
\end{gather*}
we obtain that $\mathbf{1}(\eta\ge1)(|\mathcal{G}-\Pi_2|_{ii})\prec p^{-1/2}$ for all $1 \leq i \leq n.$ This yields that $\mathbf{1}(\eta\ge 1)\Lambda(z)\prec p^{-1/2}.$ It remains to deal with the small $\eta<1$. The following weak bound is a counterpart of Lemma 5.12 of \cite{yang2019edge}.

\begin{lemma}[Weak entrywise local law]\label{lem: weak entrywise local law}
    Suppose assumptions in Theorem \ref{thm_main_locallaws} hold. Then we have  $\Lambda(z)\prec(p\eta)^{-1/4}$ uniformly for $z\in\mathbf{D}$.
\end{lemma}
\begin{proof}
One can prove this lemma using a continuity argument. The key inputs are the estimates in the $\eta \geq 1$ case. All
the other parts of the proof are essentially the same as Lemma 5.12 of \cite{yang2019edge}. We omit the details. 
\end{proof}

The following lemma is the last component for the proof of (\ref{eq_entrywise}) which is a counterpart of Lemma 5.13 of \cite{yang2019edge}.  
\begin{lemma}[Fluctuation Averaging]\label{lem: fluctuation averaging}
    Suppose assumptions in Theorem \ref{thm_main_locallaws} hold. Let $\nu\in[1/4,1]$. Denote $\Psi_{\nu}:=\sqrt{\frac{\operatorname{Im}m_{1n}+(p\eta)^{-\nu}}{p\eta}}+\frac{1}{p\eta}$. Suppose moreover that $\Lambda\prec(p\eta)^{-\nu}$ uniformly for $z\in\mathbf{D}$. Then we have that $|\mathbf{Z}_1|+|\mathbf{Z}_2|\prec\Psi_{\nu}^2.$
\end{lemma}
\begin{proof}
    Notice that $\Xi$ holds with high probability from Lemma \ref{lem: weak entrywise local law}. We can write
    \begin{gather*}
        \mathbf{Z}_1=\frac{1}{p}\sum_{j=1}^n(1-\mathbb{E}_j)\Big[\frac{\xi^2_j}{1+\xi^2_jm_1^{(j)}(z)}\mathbf{v}_j^{*}G^{(j)}\mathbf{v}_j\Big],\quad \mathbf{Z}_2=\frac{1}{p}\sum_{j=1}^n(1-\mathbb{E}_j)\Big[\frac{\xi^2_j}{(1+\xi^2_jm_1^{(j)}(z))^2}\mathbf{v}_j^{*}G^{(j)}\mathbf{v}_j\Big].
    \end{gather*}
Now the method to prove the fluctuation averaging for the random terms under conditional expectation is relatively standard; for example see Lemma 5.13 of \cite{yang2019edge} or Proposition 4.11 of \cite{Wen2021}. We omit the further details here.
\end{proof}

With the above preparation, we prove Theorem \ref{thm_main_locallaws}. We start with the entry-wise local law as in (\ref{eq_entrywise}) following that of Proposition 5.8 of \cite{yang2019edge}. 
\begin{proof}[\bf Proof of (\ref{eq_entrywise})]
 By Lemma \ref{lem: weak entrywise local law},  the event $\Xi$ holds with high probability. According to Lemmas \ref{lem: entrywiselocallaw_lem1} and \ref{lem: weak entrywise local law}, we may take $\nu=1/4$. Then, \eqref{eq: locallaw_finerest_F} reads that $\mathbf{1}(\Xi)|F_p(m_1(z),z)|\prec (\operatorname{Im}m_{1n}(z)+(p\eta)^{-1/4})(p\eta)^{-1}.$ Using the results in Theorem \ref{lem: stability of F}, one has
\begin{gather*}
    |m_1-m_{1n}|\prec\frac{\operatorname{Im}m_{1n}}{p\eta\sqrt{\kappa+\eta}}+\frac{1}{(p\eta)^{5/8}}\prec \frac{1}{(p\eta)^{5/8}},
\end{gather*}
where we used $\operatorname{Im}m_{1n}(z)=\rO(\sqrt{\kappa+\eta})$ from \eqref{eq: asymlaws_main_estm12n}. Similar bound can be derived for $|m_2-m_{2n}|$ by \eqref{eq: locallaw_finerest_m}. Then we get an updated bound $\Theta\prec(p\eta)^{-5/8}.$ Plugging such updated bound in $\Psi_{\Theta}$ in Lemma \ref{lem: entrywiselocallaw_lem1}, one has 
\begin{gather*}
    \Lambda_o\prec\sqrt{\frac{\operatorname{Im}m_{1n}(z)+(p\eta)^{-5/8}}{p\eta}}+\frac{1}{p\eta},
\end{gather*}
uniformly in $z\in\mathbf{D}$. Up to now, we obtain a better bound for $\Lambda_o$. Iterating the above arguments as in the discussions between (5.57) and (5.58) of \cite{yang2019edge}, we get the bound 
\begin{gather}\label{eq: optimal bound for Theta}
    \Theta\prec (p\eta)^{-1}.
\end{gather}
Combining \eqref{eq: expan_cal G} with \eqref{eq: optimal bound for Theta} and Lemma \ref{lem: fluctuation averaging}, we can finish the proof. 
\end{proof}

Then we prove the averaged local laws as in (\ref{eq_averagedone}) and (\ref{eq_averagedtwo}) following that of Proposition 5.1 of \cite{yang2019edge}. 
\begin{proof}[\bf Proof of (\ref{eq_averagedone}) and (\ref{eq_averagedtwo})]
By a discussion similar to (\ref{eq_m1}), we can obtain that 
\begin{gather}\label{eq: expression for m}
    m(z)=-\frac{1}{p}\sum_{i=1}^p\frac{1}{z(1+m_2(z)\sigma_i)}+\rO_{\prec}(\mathbf{Z}_2+\Psi_{\Theta}^2).
\end{gather}
Using \eqref{eq_systemequationsm1m2elliptical}, \eqref{eq: asymlaws_main_lowerbound}, Lemma \ref{lem: fluctuation averaging} and \eqref{eq: optimal bound for Theta}, we can prove (\ref{eq_averagedone}). Moreover, for $z\in\mathbf{D}\cap\{z=E+\mathrm{i}\eta:E\ge\lambda_{+},\;p\eta\sqrt{\kappa+\eta}\ge p^{\varepsilon_e}\}$, using  \eqref{eq: asymlaws_main_estm12n}, we observe that 
\begin{gather}
    \Psi_{\Theta}^2\le2\Big(\frac{\operatorname{Im}m_{1n}(z)}{p\eta}+\frac{1}{(p\eta)^2}\Big)\lesssim\frac{1}{p\sqrt{\kappa+\eta}}+\frac{1}{(p\eta)^2}.
\end{gather}
Taking $\nu=1$ in Lemma \ref{lem: fluctuation averaging}, we can prove (\ref{eq_averagedtwo}) using Lemma \ref{lem: stability of F}. This completes the proof. 
\end{proof}

\bibliographystyle{abbrv}
\bibliography{Bib}

\begin{thebibliography}{10}

\bibitem{BaoPanandZhou2015}
Z.~Bao, G.~Pan, and W.~Zhou.
\newblock Universality for the largest eigenvalue of sample covariance matrices
  with general population.
\newblock {\em The Annals of Statistics}, 43(1):382--421, 2015.

\bibitem{bao2022extreme}
Z.~Bao and X.~Xu.
\newblock Extreme eigenvalues of log-concave ensemble.
\newblock {\em arXiv preprint arXiv:2212.11634}, 2022.

\bibitem{Ding2023}
X.~Ding, J.~Xie, L.~Yu, and W.~Zhou.
\newblock Extreme eigenvalues of sample covariance matrices under generalized
  elliptical models with applications.
\newblock {\em arXiv preprint arXiv 2303.03532}, 2023.

\bibitem{DingandYang2018}
X.~Ding and F.~Yang.
\newblock A necessary and sufficient condition for edge universality at the
  largest singular values of covariance matrices.
\newblock {\em The Annals of Applied Probability}, 28(3):1679--1738, 2018.

\bibitem{ding2021spiked}
X.~Ding and F.~Yang.
\newblock Spiked separable covariance matrices and principal components.
\newblock {\em The Annals of Statistics}, 49(2):1113--1138, 2021.

\bibitem{ding2022tracy}
X.~Ding and F.~Yang.
\newblock Tracy-{W}idom distribution for heterogeneous {G}ram matrices with
  applications in signal detection.
\newblock {\em IEEE Transactions on Information Theory}, 68(10):6682--6715,
  2022.

\bibitem{Elkaroui2007}
N.~El~Karoui.
\newblock Tracy-{W}idom limit for the largest eigenvalue of a large class of
  complex sample covariance matrices.
\newblock {\em Annuals of Probability}, 35:663--714, 2007.

\bibitem{el2009concentration}
N.~El~Karoui.
\newblock Concentration of measure and spectra of random matrices: applications
  to correlation matrices, elliptical distributions and beyond.
\newblock {\em The Annals of Applied Probability}, 19(6):2362--2405, 2009.

\bibitem{fang1990}
K.~T. Fang and T.~W. Anderson.
\newblock {\em Statistical inference in elliptically contoured and related
  distributions}.
\newblock Allerton Press, 1990.

\bibitem{8704905}
J.~Hu, W.~Li, and W.~Zhou.
\newblock Central limit theorem for mutual information of large {MIMO} systems
  with elliptically correlated channels.
\newblock {\em {IEEE} Transactions on Information Theory}, 65(11):7168--7180,
  2019.

\bibitem{10.1214/aos/1009210544}
I.~M. Johnstone.
\newblock {On the distribution of the largest eigenvalue in principal
  components analysis}.
\newblock {\em The Annals of Statistics}, 29(2):295 -- 327, 2001.

\bibitem{kwak2021extremal}
J.~Kwak, J.~O. Lee, and J.~Park.
\newblock Extremal eigenvalues of sample covariance matrices with general
  population.
\newblock {\em Bernoulli}, 27(4):2740--2765, 2021.

\bibitem{lee2015edge}
J.~O. Lee and K.~Schnelli.
\newblock Edge universality for deformed wigner matrices.
\newblock {\em Reviews in Mathematical Physics}, 27(08):1550018, 2015.

\bibitem{lee2016extremal}
J.~O. Lee and K.~Schnelli.
\newblock Extremal eigenvalues and eigenvectors of deformed {W}igner matrices.
\newblock {\em Probability Theory and Related Fields}, 164(1-2):165--241, 2016.

\bibitem{LS}
J.~O. Lee and K.~Schnelli.
\newblock {Tracy–{W}idom distribution for the largest eigenvalue of real
  sample covariance matrices with general population}.
\newblock {\em The Annals of Applied Probability}, 26(6):3786 -- 3839, 2016.

\bibitem{LSSY}
J.~O. Lee, K.~Schnelli, B.~Stetler, and H.-T. Yau.
\newblock {Bulk universality for deformed Wigner matrices}.
\newblock {\em The Annals of Probability}, 44(3):2349 -- 2425, 2016.

\bibitem{LeenandYin2012}
J.~O. Lee and J.~Yin.
\newblock A necessary and sufficient condition for edge universality of
  {W}igner matrices.
\newblock {\em Duke Mathematical Journal}, 163, 06 2012.

\bibitem{PaulandSilverstein2009}
D.~Paul and J.~Silverstein.
\newblock No eigenvalues outside the support of the limiting empirical spectral
  distribution of a separable covariance matrix.
\newblock {\em Journal of Multivariate Analysis}, 100:37--57, 01 2009.

\bibitem{PillaiandYin2014}
N.~S. Pillai and J.~Yin.
\newblock Universality of covariance matrices.
\newblock {\em The Annals of Applied Probability}, 24(3):935--1001, 2014.

\bibitem{Wen2021}
J.~Wen, J.~Xie, L.~Yu, and W.~Zhou.
\newblock Tracy-{W}idom limit for the largest eigenvalue of high-dimensional
  covariance matrices in elliptical distributions.
\newblock {\em Bernoulli}, 28(4):2941--2967, 2022.

\bibitem{yang2019edge}
F.~Yang.
\newblock Edge universality of separable covariance matrices.
\newblock {\em Electronic Journal of Probability}, 24:1--57, 2019.

\end{thebibliography}

\end{document}